\documentclass[onefignum,onetabnum]{siamart190516}

\RequirePackage{fix-cm}
\usepackage{graphicx}
\usepackage{psfrag}
\usepackage{subfigure}
\usepackage{url}
\usepackage{color}
\usepackage{cite}
\usepackage{epsfig}
\usepackage{multirow}
\usepackage{array}
\usepackage{amsmath,amssymb}
\usepackage{algorithmic}
\usepackage{comment}

\newcommand{\rset}{\mathbb{R}}

\newcommand{\Eb}{\mathbb{E}}

\newcommand{\red}{\textcolor{black}}

\DeclareMathOperator*{\argmin}{arg\,min}

\newsiamremark{remark}{Remark}
\newtheorem{assumption}[theorem]{Assumption}

\newsiamremark{example}{Example}

\title{	A systematic approach to general higher-order majorization-minimization algorithms for (non)convex optimization}

\author{Ion Necoara\thanks{Automatic Control and Systems
		Engineering Department, National University of Science and Technology  Politehnica Bucharest, Romania and   Gheorghe Mihoc-Caius Iacob Institute of Mathematical Statistics and Applied Mathematics of the Romanian Academy. \email{ion.necoara@upb.ro.}}
	\and Daniela Lupu\thanks{Automatic Control and Systems
		Engineering Department, National University of Science and Technology  Politehnica Bucharest, Romania. \email{daniela.lupu@upb.ro.}}}

\begin{document}

\maketitle

\begin{abstract}
	Majorization-minimization algorithms consist of successively minimizing a sequence of upper bounds of the objective function so that along the iterations the objective  decreases. Such a simple principle  allows to solve a large class of optimization problems,  even nonconvex and nonsmooth. We propose a general  higher-order majorization-minimization algorithmic framework for minimizing an objective function that admits an approximation (also called surrogate)  such that the corresponding error function has  a higher-order Lipschitz continuous  derivative. We present convergence guarantees for our new method for general optimization problems with (non)convex and/or (non)smooth objective function.  For  convex (possibly nonsmooth) problems we provide global sublinear convergence rates,  while for problems with uniformly convex objective function we obtain  faster  local  superlinear convergence rates.   We also prove global   first-order optimality conditions guarantees and sublinear convergence rates for general nonconvex (possibly nonsmooth)  problems and under Kurdyka-Lojasiewicz property of the objective we derive local convergence rates ranging from sublinear to superlinear for our majorization-minimization algorithm. Finally, for unconstrained  nonconvex problems we derive convergence rates in terms of  first- and second-order optimality conditions.
\end{abstract}




\section{Introduction}
The principle of successively minimizing upper bounds of the objective function is often called \textit{majorization-minimization} \cite{LanHun:00, Mai:15, RazHon:13}. Most techniques, e.g.,  gradient descent, utilize convex quadratic majorizers based on first-order oracles  in order to guarantee that the majorizer is easy to minimize.  Despite the  empirical success of first-order majorization-minimization algorithms  to solve difficult optimization problems, the convergence speed of such methods  is known to slow down close to saddle points or in ill-conditioned landscapes \cite{Nes:04,Pol:87}. Higher-order methods are known to be less affected by these problems \cite{BriGar:17, Nes:19, cartis2017,cartis:2020}. 
In this work, we  focus our attention on  higher-order majorization-minimization methods to  find (local) minima of   (potentially  nonsmooth and nonconvex) objective functions. At each iteration, these algorithms construct and optimize a (usually Taylor-based)  local model of the objective using higher-order derivatives with an additional penalty term that enforces the model  to approximate  well  the real objective locally. 

\medskip 

\noindent \textbf{Contributions}.  This paper provides a unified algorithmic framework  based on the notion of higher-order upper bound approximations of the  (non)convex and/or (non)smooth objective function, leading to a    \textit{general higher-order majorization - minimization}  algorithm, which we call GHOM.   Then,    we present convergence guarantees for the GHOM algorithm for  general  optimization problems when the upper bounds approximate the objective function  up to an error that is $p \geq 1$ times  differentiable and has a Lipschitz continuous $p $ derivative; we call such upper bounds \textit{higher-order surrogate} functions.  More precisely,  on general convex (possibly nonsmooth)  problems GHOM algorithm  achieves  global sublinear convergence rate  for the  function values. When we apply our method to  optimization problems with uniformly convex objective function, we obtain faster  local superlinear  convergence rates  in multiple criteria: function values, distance of the iterates to the optimal point  and  norm of  subgradients.  The last criterion also implies convergence rates for the residual of the first-order optimality conditions.  Then, for general nonconvex  (possibly nonsmooth) problems we  prove for GHOM global asymptotic   convergence to a point from which there is no descent direction as well as  convergence rates for the residual  of the first-order optimality conditions.  We also characterize the convergence rates of  GHOM algorithm  locally in terms of function values under the Kurdyka-Lojasiewicz (KL) property.  Our results show that the convergence behavior  of GHOM ranges from sublinear to superlinear   depending on the parameter of the underlying KL geometry. Finally, on  smooth unconstrained nonconvex problems  we derive convergence rates in terms of first- and second-order optimality conditions. 

\medskip 

\noindent The folllowing remarks clarify better our contributions. First, it is known that for general nonsmooth, nonconvex problems, answering to the question of whether a descent direction exists from a point is NP-hard \cite{Nes:13}.  However, for nonconvex problems with composite structure (i.e., the objective function formed as a sum of two terms,  one is smooth and another is simple and convex) Nesterov proves for some first-order schemes convergence  to a point from which there is no descent direction \cite{Nes:13} (see also \cite{cartis:2020} for  a similar result on problems with smooth objective and simple constraints). \textit{In this paper, we derive global convergence results  (including convergence rates) for GHOM to a point from which there is no descent direction for  general nonconvex problems. The only condition  we require is that the objective function admits a higher-order surrogate (see Definition \ref{def:sur}).}   

\medskip    

\noindent Second, besides providing a \textit{unifying  framework} for the design and analysis of  higher-order majorization-minimization  methods, in special cases, where complexity bounds are known for some particular higher-order algorithms, our convergence results recover the  existing bounds.   E.g., our rates recover for $p=1$ the  convergence bounds for the (proximal) gradient  \cite{Mai:15,AttBol:09,RazHon:13} and Gauss-Newton  \cite{DruPaq:19, Pau:16}  type algorithms from the literature.  For $p>1$ and convex composite objective functions (see Example \ref{expl:5}), we recover the sublinear convergence results from \cite{Nes:19Inxt} and the local superlinear convergence  rates  from \cite{DoiNes:2019}.  For  $p>1$ and  nonconvex unconstrained optimization problems (see Example \ref{expl:4}), we recover the convergence results from \cite{BriGar:17,cartis2017} and for smooth  problems with simple constraints we obtain similar rates as in  \cite{cartis:2020}.  \textit{However, for other examples (such as, Examples \ref{expl:88},  \ref{expl:6}, in the convex case; Examples \ref{expl:8}, \ref{expl:88}, \ref{expl:5},  \ref{expl:888} and  \ref{expl:6} in the  nonconvex case) our convergence results seem to be new}.    

\medskip  

\noindent Finally,  our unifying algorithmic framework is inspired in part by the recent work  on higher-order Taylor-based  methods for  convex optimization  \cite{Nes:19, Nes:19Inxt}, but it yields a more general update rule  and it is also appropriate for nonconvex optimization. Note that there  is a major difference between the Taylor expansion and the model approximation based on a general majorization-minimization framework. Taylor expansion is unique.  Conversely, majorization-minimization approach may admit many upper bound models for a given objective  and every model leads to a different optimization method. Another major difference between our approach and the existing works such as  \cite{AttBol:09, BriGar:17, DruPaq:19, Nes:19, Nes:19Inxt, cartis2017,cartis:2020} is that  we assume Lipschitz continuity of the $p$ derivative of the error function, while the other papers   assume directly the Lipschitz continuity of the $p$ derivative of the objective.  \textit{Hence, our framework leads  to an elegant and general convergence analysis  in both convex and nonconvex settings, with convergence proofs  different from the existing works, and yielding new convergence results and rates  or covering results  that are otherwise scattered across  a dozen papers.}


\subsection{Related work}
\textbf{Higher-order methods} are popular due to their performance in dealing with ill conditioning and fast  rates of convergence \cite{BriGar:17, CarGou:11, cartis2017, cartis:2020, GraNes:19,  NesPol:06, Nes:19, Nes:19Inxt}.  For example, in \cite{Nes:19} the following  unconstrained convex  problem was considered:
\begin{equation}\label{sota:eq1}
	f^* = \min_{x \in \rset^n} f(x),
\end{equation}
\noindent where $f$ is convex, $p$ times continuously differentiable and with the $p$ derivative Lipschitz continuous  of constant $L_p^f$ (see Section \ref{sec:notation} for a precise definition). Then, Nesterov proposed  in \cite{Nes:19} the following higher-order Taylor-based iterative method for finding an optimal solution of the convex problem  \eqref{sota:eq1}: 
\vspace*{-0.2cm}
\begin{equation}
	\label{sota:eq2}
	x_{k+1} \!=\! \argmin_{y \in \rset^n}  g(y;x_k)   \left(\!:= \!\sum_{i=0}^{p} \!\frac{1}{i !} \nabla^{i} f(x_k)[y \!-\! x_k]^{i} \!+\! \frac{M_p}{(p+1)!}\|y \!-\! x_k\|^{p+1} \! \right),
\end{equation}
\noindent where $M_p \geq L_p^f$. Hence, at each iteration  one needs to construct and optimize a local Taylor model of the objective  with an additional regularization term that depends on how well the model approximates the real objective.  This is a natural extension of the cubic regularized Newton's method extensively analyzed e.g., in \cite{NesPol:06, CarGou:11, ZhoWan:18}. Under the above settings  \cite{Nes:19}  proves the following  convergence rate in   function  values for 	\eqref{sota:eq2}:
\[
f\left(x_{k}\right)-f^{*} = \mathcal{O}\left(\frac{1}{k^p}\right) \quad \forall  k \geq 1.
\]

\noindent  Extensions of this method to  composite convex problems and to objective functions with Holder continuous higher-order derivatives have been given in \cite{Nes:19Inxt} and \cite{GraNes:19}, respectively,  inexact variants were analyzed e.g., in \cite{NesDoi:20}, higher-order proximal point variants were given  e.g., in \cite{Nes:21}  and local superlinear convergence results were given recently in \cite{DoiNes:2019}. 

\medskip  

\noindent Further, for  the unconstrained nonconvex case \cite{BriGar:17,cartis2017} provides  convergence rates for a similar algorithm. The basic assumptions are that $f$ is $p$ times continuously differentiable having the $p$ derivative smooth with constant $L_p^f$ and bounded below. For example, \cite{cartis2017} proposes an adaptive regularization  algorithm (called AR$p$), which requires building  a higher-order model based on an appropriate regularization of the $p$  Taylor approximation $g(y;x_k)$.
Basically, at each iteration the  AR$p$ algorithm    approximately computes a (local) minimum of the  model $g(y;x_k)$ that must satisfy certain second-order optimality conditions:
$	x_{k+1} \approx  \argmin_{y \in \mathbb{R}^n} g(y; x_k)$.
For the AR$p$ algorithm,  \cite{cartis2017} proves  the best known convergence rate for this class of problems ($f$ nonconvex with $p>1$ derivative smooth) in terms of the  first- and second-order optimality conditions:  
\vspace*{-0.2cm}
\begin{align*} 
	\min_{i=1:k} \max  &  \left(- \lambda_{\min}^{\frac{p+1}{p-1}}(\nabla^2 f(x_{i})),\;  	
	\| \nabla f(x_{i})\|_{*}^{\frac{p+1}{p}} \right) =\mathcal{O}\left(\frac{1}{k}\right).
\end{align*}

\vspace*{-0.2cm}

\noindent Extension of these results to smooth optimization problems with simple constraints were  given recently in  \cite{cartis:2020}.  Furthermore, several studies  demonstrate that special geometric properties of the objective function, such as gradient dominance condition or Kurdyka-Lojasiewicz (KL) property \cite{BolDan:07}, can enable faster convergence of the ARp algorithm (e.g.,  when $p=2$), see  for example   \cite{NesPol:06,ZhoWan:18}. 

\medskip 

\noindent Notably, all these approaches require optimizing a $(p+1)$-th order polynomial, which is known to be a difficult problem. Recently, \cite{Nes:19, Nes:19Inxt} introduced an implementable method  for convex functions (see also Lemma \ref{lema:conv} below). In particular, for $p=2$ and $p=3$ there are efficient methods from convex optimization for minimizing the corresponding Taylor-based upper approximation \eqref{sota:eq2}, see e.g., \cite{CarGou:11, CarDuc:16, GraNes:19sub, Nes:19, Nes:19Inxt} . 

\medskip 

\noindent \textbf{Majorization-minimization algorithms}  approximate at each iteration the objective function by a  majorizing function that is easy to minimize \cite{BolPau:16, LanHun:00}.  Most techniques, e.g.,  gradient descent, utilize convex quadratic majorizers in order to guarantee that the model at each iteration is easy to minimize.  The framework of first-order majorization-minimization methods, i.e. methods that are using only gradient information to build the upper model,   has been analyzed  widely in the literature of the recent two decades, see e.g.,  \cite{BolPau:16, HunLan:04, LanHun:00, Mai:15, RazHon:13}. \textit{However, to the best of our knowledge there are no results on the convergence behavior of general  higher-order majorization-minimization algorithms, i.e. methods that are using higher-order derivatives  to build the upper model.}  In this paper we  provide a general framework for the design of  higher-order majorization-minimization  algorithms and derive a complete convergence analysis for them (global and local convergence rates) covering a large class of optimization problems, that is convex or nonconvex, smooth or nonsmooth.

\medskip 

\noindent \textbf{Content}. The paper is organized as follows: Section \ref{sec:notation} presents notation and preliminaries; in Section \ref{sect:ghom} we define our higher-order majorization-minimization framework and the corresponding algorithm;   in Section \ref{sec:conv} we derive  convergence results for our scheme in the convex settings, while the convergence analysis for nonconvex problems is given in Section~\ref{sec:non-conv}.


\subsection{Notations and  preliminaries} 
\label{sec:notation}
We denote a finite-dimensional real vector space with $\mathbb{E}$ and by $\mathbb{E}^*$ its dual space composed by linear functions on $\mathbb{E}$. For such a function $s \in  \mathbb{E}^*$, we denoted by $\langle s, x\rangle$ its value at $x \in \mathbb{E}.$ Utilizing a self-adjoint positive-definite operator $B: \mathbb{E} \rightarrow \mathbb{E}^*$ (notation $B = B^{*} \succ 0$), we can endow  these spaces with\textit{ conjugate Euclidean norms}:
$$\|x\|=\langle B x, x\rangle^{1 / 2} \quad \forall x \in \mathbb{E}, \quad\|y\|_{*}=\left\langle y, B^{-1} y\right\rangle^{1 / 2}  \quad \forall y \in \mathbb{E}^{*}.$$
\noindent   For a convex set $\mathcal{X} \subset \mathbb{E}$, the normal cone of $\mathcal{X}$ at $\bar x \in \mathcal{X}$ is defined as $\mathcal{N}_{\mathcal{X}}(\bar x) = \{ g:  \langle g,  \bar x - x\rangle \geq 0 \; \; \forall x \in \mathcal{X}   \}$.    For an extended real-valued  function  $\psi:   \mathbb{E}  \rightarrow \bar{\mathbb{R}}  \, (:= \mathbb{R} \cup \{\infty\}$), we always assume that it is proper (i.e., $\psi \not= \infty$), lower semicontinuous (lsc) and its domain is $\text{dom} \psi=\{x \in \mathbb{E}: \; \psi(x) < \infty \}$.  The regular subdifferential of $\psi$ at $\bar x \in \text{dom} \psi$ is defined as \cite{Roc:70}:  
$$\hat \partial  \psi(\bar x) =\{ \psi^{\bar x}:  \liminf\limits_{\substack{x \to \bar{x}, x \neq \bar{x}}}  \frac{\psi(x) - \psi(\bar x) - \langle \psi^{\bar x}, x - \bar x \rangle  }{\|x -\bar x\|} \geq 0.$$  
Further, the limiting subdifferential of $\psi$ at $\bar x$ is defined as the sequential outer limit 
$$\partial \psi(\bar x) = \limsup_{\substack{x \xrightarrow[]{\psi} \bar{x}}}  \hat \partial  \psi(x),$$ where the notation $\substack{x \xrightarrow[]{\psi} \bar{x}}$ signifies that $x \to \bar x$ with $\psi(x) \to \psi(\bar x)$. Note that $\partial \psi(\bar x) \not= \emptyset$ if $\psi$ is locally Lipschitz around $\bar x$.  Moreover, for $\psi$ convex function, $\partial \psi(x)$  denotes its subdifferential at $x$, which  is a closed set.  Let us  denote:
$$S(x) = \text{dist}(0, \partial \psi(x)) \;  \left(:= \inf_{\psi^x \in \partial \psi(x)} \| \psi^x\|_* \right)  \quad \forall x \in \text{dom} \, \psi. $$  
If  $\partial \psi(x) = \emptyset$, we set $S(x) = \infty$.  For a smooth function  $\psi:   \mathbb{E}  \rightarrow \bar{\mathbb{R}}$,  which is $p\geq 1$ times continuously differentiable on the convex and open domain $\text{dom}  \psi =\{ x \in \mathbb{E}: \psi(x) < \infty\}$,  denote by $\nabla \psi(x)$  its gradient and  $\nabla^2 \psi(x)$ its Hessian   at the point $x \in \text{dom} \psi$. Note that  $\nabla \psi(x) \in \mathbb{E}^* $  and $  \nabla^2 \psi(x) d \in \mathbb{E}^*$ for all $d \in \mathbb{E}$.  In what follows, we often work with directional derivatives:   $D^p \psi(x) \left[d_{1}, \ldots, d_{p}\right]$ denotes	  the $p$ directional derivative of function $\psi$ at $x \in \text{dom} \psi$ along directions  $d_1, \cdots, d_p \in \mathbb{E}$ (see also \cite{Nes:19, Nes:19Inxt} for a similar exposition).  For example, for a twice differentiable function $\psi$ one has for any $x \in \text{dom}  \psi $ and $d, \bar d \in \mathbb{E}$ that $D^1 \psi(x)[d] = \langle  \nabla \psi(x), d\rangle$ and $D^2 \psi(x)[d,\bar d] = \langle  \nabla^2 \psi(x) d, \bar d\rangle$. On the other hand, in this paper we consider  for a nonsmooth function $\psi$  the first directional derivative  classically defined as: $D \psi(x)[d] = \lim_{t \downarrow 0} \left(  \psi(x + td) - \psi(x) \right)/t$ (for simplicity, we omit the index $p$ for $p=1$ and use $D \psi(x)[d]$ instead of $D^1 \psi(x)[d]$).  
Note that   $ D^p \psi(x) [\cdot]$ is  a symmetric  $p$  multilinear form on $\mathbb{E}$.  The abbreviation $D^p \psi(x)[d]^{p}$
	is used when all directions are the same, i.e., $d_{1}=\cdots=d_{p}=d$ for some $d \in \mathbb{E}$. The norm of $D^p \psi(x)$ is defined in the standard way (see \cite{Nes:19,Nes:19Inxt}):
	\begin{equation*} 
		\|D^p \psi(x) \|:=\max _{\left\|d_{1}\right\|,\cdots,\left\|d_{p}\right\| \leq 1}\left|D^p \psi(x)\left[d_{1}, \ldots, d_{p}\right]\right|  =  \max _{\left\|d\right\| \leq 1}\left|D^p \psi(x)\left[d\right]^p\right|   \text{.}
	\end{equation*}

\noindent Since for any fixed $x, y \in \text{dom} \psi$ the form  $D^p \psi(x) [\cdot] - D^p \psi(y) [\cdot]$ is also $p$ multilinear and symmetric,  then we can define the following class of \textit{$p$ smooth functions}: 

\begin{definition} Let $\psi: \mathbb{E} \rightarrow \mathbb{R}$ be $p\geq 1$ times continuously differentiable. Then, the $p$ derivative of $\psi$ is Lipschitz continuous if there exist $L_p^{\psi} > 0$  such that
	\begin{equation} \label{eq:1}
		\| D^p \psi(x) - D^p \psi(y) \| \leq L_p^{\psi} \| x-y \| \quad  \forall x,y \in \text{dom} \psi.
	\end{equation}
\end{definition}

\noindent We denote the Taylor approximation of $\psi$ around $x \in \text{dom} \psi$ of order $p$ by:
\vspace*{-0.2cm}
$$
T_p^{\psi}(y;x)= \psi(x) + \sum_{i=1}^{p} \frac{1}{i !} D^{i} \psi(x)[y-x]^{i}  \quad \forall y \in E.
$$		
\noindent It is known that if \eqref{eq:1} holds, then by the standard integration arguments  the residual between function value and its Taylor approximation can be bounded \cite{Nes:19}:
\vspace*{-0.2cm}
\begin{equation}\label{eq:TayAppBound}
	|\psi(y) - T_p^{\psi}(y;x) | \leq  \frac{L_p^{\psi}}{(p+1)!} \|y-x\|^{p+1}  \quad  \forall x,y \in \text{dom} \psi. 
\end{equation}

\vspace*{-0.2cm}

\noindent Applying the same reasoning for the functions $\langle \nabla \psi(x), d\rangle$ and $\langle \nabla^2 \psi(x) d, d\rangle$, with direction $d \in \Eb$ being fixed,   we also get the following inequalities valid for all $ x,y \in \text{dom} \psi$ and $p \geq 2$, see also \cite{Nes:19,Nes:19Inxt}:
\vspace*{-0.2cm}
\begin{align} \label{eq:TayAppG1}
	&\| \nabla \psi(y) - \nabla T_p^{\psi}(y;x) \|_* \leq \frac{L_p^{\psi}}{p!} \|y-x \|^p, \\
	\label{eq:TayAppG2}
	&\|\nabla^2 \psi(y) - \nabla^2 T_p^{\psi}(y;x) \| \leq \frac{L_p^{\psi}}{(p-1)!} \| y-x\|^{p-1}.
\end{align}

\vspace*{-0.2cm}

\noindent For the Hessian we consider  the spectral norm of self-adjoint linear operators (maximal module of all eigenvalues computed with respect to operator $B$).  Next, we provide several examples of functions that have known Lipschitz continuous $p$ derivatives. 


\begin{example}
	\label{expl:1}
		Given $x_0 \in \mathbb{E}$ {and a self-adjoint positive-definite operator  $B \succ 0 $, defining the  norm $\|x\|=\langle B x, x\rangle^{1 / 2} $ for all $ x \in \mathbb{E}$,}  then the power  norm $\psi(x)= \left\|x-x_{0}\right\|^{p+1}$, with $p \geq 1$, has the $p$ derivative Lipschitz continuous  with Lipschitz  constant $L_{p}^{\psi}=(p+1)!$, see Theorem 7.1 in  \cite{NesRod:20}.  
\end{example}


\begin{example}
	\label{expl:2} 
	For given $a_{i} \in \mathbb{E}^{*}, 1 \leq i \leq m,$ consider the log-sum-exp function:
	\vspace*{-0.2cm}
	\[  \psi(x)=\log \left(\sum_{i=1}^{m} e^{\left\langle a_{i}, x\right\rangle}\right), \quad x \in \mathbb{E}.  \]
	
	\vspace*{-0.2cm}
	
	\noindent Note that for $m=2$ and  $a_1 =0$, we recover the logistic regression loss function,  widely used in machine learning \cite{Mai:15}.  Furthermore, for the Euclidean norm $\|x\|=\langle B x, x\rangle^{1 / 2}$ for $ x \in \mathbb{E}$ and $B:=\sum_{i=1}^{m} a_{i} a_{i}^{*}$ (assuming $B \succ 0,$ otherwise we can reduce dimensionality of the problem), the Lipschitz continuous condition	\eqref{eq:1} holds with $L_{1}^\psi=1$ for $p=1$, $L_{2}^\psi=2$ for $p=2$ and $L_{3}^\psi=4$ for $p=3$, see e.g.,  \cite{DoiNes:2019}.	
\end{example}


\begin{example}
	\label{expl:3} 
	For a $p \geq 2$ times differentiable function $\psi$ its Taylor approximation of order $p$ at a given point $x$,   $T_p^{\psi}(\cdot;x)$,  has the  $p-1$ derivative Lipschitz continous with Lipschitz  constant  $L_{p-1}^{T_p^{\psi}} = \|D^p \psi(x)\|$.  
\end{example}

\medskip 

\noindent Finally, in the convex case Nesterov proved in \cite{Nes:19} a remarkable result which states  that an appropriately regularized Taylor approximation of a convex function is a convex multivariate polynomial. 
\begin{lemma} \cite{Nes:19}
	\label{lema:conv}
	Assume $\psi$  convex and $p>2$ differentiable  function having  the $p$ derivative Lipschitz  continuous  with constant $L_p^\psi$ on $\Eb$.  Then, the regularized Taylor approximation:

\vspace{-0.5cm}

	\begin{equation*} 
		g(y,x) = T_p^{\psi}(y;x) + \frac{M_p}{(p+1)!}\| y-x\|^{p+1} 
	\end{equation*}

\vspace{-0.2cm}

\noindent 	is also a convex function  w.r.t.   $y$, provided that $M_p \geq pL_p^{\psi}$,   i.e.,  we have $\nabla^2_y g(y;x) - \nabla^2 \psi(y) \succeq 0$ for all $y \in \Eb$. 
\end{lemma}
As discussed in introduction,   in higher-order  methods one usually needs to minimize at each iteration a regularized higher-order  Taylor approximation.  
Therefore,  when minimizing convex functions with higher-order algorithms one can use  a large number of powerful methods from  convex optimization for solving the subproblem \cite{CarGou:11, CarDuc:16, GraNes:19sub, Nes:19Inxt}.  Further, let us introduce the class of  uniformly convex functions \cite{Nes:19Inxt,NesPol:06} which will play a key role in proving the superlinear convergence  of the iterates of our algorithm. 

\begin{definition}
	\noindent A function $\psi : \Eb \to \bar{\mathbb{R}}$ is \textit{uniformly convex} of degree $q \geq 2$ on $\text{dom} \;\psi$ if there exists a positive constant  $\sigma_{q}>0$ such that: 
	\begin{equation}
		\label{eq:unifConv}
		\psi(y) \geq \psi(x)+\left\langle \psi^{x}, y-x\right\rangle+\frac{\sigma_{q}}{q}\|x-y\|^{q} \quad   \forall x, y \in \text{dom} \, \psi, 
	\end{equation}
\noindent 	where $\psi^{x}$ is an arbitrary vector from the subdifferential $\partial \psi(x)$ at $x$. 
\end{definition}

\noindent Next, we establish a relation between uniform convexity of a function and its minimum over a convex set. 
\begin{lemma}
	\label{lema:uc}
Let $\psi : \Eb \to \bar{\mathbb{R}}$ be a  \textit{uniform convex} function satisfying  \eqref{eq:unifConv} and $\mathcal{X} \subseteq  \text{dom} \, \psi$ a closed convex set. Then, the function  $\psi + \textbf{1}_\mathcal{X}$  is  uniform convex on $\mathcal{X}$ and  the following relation holds: 
	\begin{align}\label{eq:probConv}
\min _{y \in  \mathcal{X}} \psi(y) \geq  	\psi(x) 	-\frac{q-1}{q}\left(\frac{1}{\sigma_{q}}\right)^{\frac{1}{q-1}}\left\|\psi_X^{x}\right\|_{*}^{\frac{q}{q-1}}   \quad \forall x \in \mathcal{X}, 
\end{align}
where $\psi_\mathcal{X}^{x} $ is an arbitrary vector from the subdifferential of the convex  function $\psi + \textbf{1}_\mathcal{X}$ at $x$, where $\textbf{1}_\mathcal{X}$ denotes the indicator function of the set $\mathcal{X}$.
\end{lemma}
\begin{proof}
Since  the function $\psi : \Eb \to \bar{\mathbb{R}}$ is  uniform convex on $\text{dom} \, \psi$ (i.e., it  satisfies   \eqref{eq:unifConv})    and the indicator function $\textbf{1}_\mathcal{X}$ is convex, adding them and using basic calculus rules for subdifferential of convex functions (i.e., $\partial \psi + \partial \textbf{1}_\mathcal{X} \subseteq \partial(\psi + \textbf{1}_X)$), we also get that the sum $\psi + \textbf{1}_\mathcal{X}$  is uniform convex on $\text{dom} (\psi + \textbf{1}_X) = \mathcal{X}$, i.e.:
\[   	(\psi + \textbf{1}_\mathcal{X}) (y) \geq (\psi + \textbf{1}_\mathcal{X})(x)+\left\langle \psi_\mathcal{X}^{x}, y-x\right\rangle+\frac{\sigma_{q}}{q}\|x-y\|^{q} \quad   \forall x, y \in \text{dom} (\psi + \textbf{1}_X) = \mathcal{X}.
\vspace*{-0.3cm}
   \]
Minimizing both sides of previous inequality w.r.t. $y$, we get  our statement.
\end{proof}

\noindent Note that for $q = 2$ in \eqref{eq:unifConv}  we recover  the usual  definition of  a strongly convex function. Moreover,   \eqref{eq:probConv} for $q = 2$  is the main property  used when analyzing the   convergence behavior of first-order methods \cite{NecNes:16}. One important class of uniformly convex functions is given next (see e.g.,  \cite{Nes:19Inxt} for a proof). 

\begin{example}
	\label{expl:uc} 
	For  $q \geq 2$ let us consider the convex function $ \psi(x) = \frac{1}{q} \|  x - x_0\|^q, $ 
	where $x_0$ is given and $B = I_n$. Then, $\psi$ is uniformly convex of degree $q$ with $\sigma_q = 2^{2-q}$.  
\end{example}

\noindent For nonconvex functions we have a more general notion than uniform convexity, called  the Kurdyka-Lojasiewicz (KL) property, which captures a broad spectrum of the local geometries that a nonconvex function can have \cite{BolDan:07}.  

\begin{definition}
	\label{def:kl}
	\noindent A proper and lower semicontinuous  function $\psi: \Eb \to \bar{\mathbb{R}}$ satisfies the uniformized  \textit{Kurdyka-Lojasiewicz (KL)} property on a compact set $\Omega \subseteq \text{dom} \, \psi$ on which $\psi$ takes a constant value $  \psi_*$ if there exist $\delta, \epsilon >0$ such that   one has:
	\vspace*{-0.2cm}
		\begin{equation*}
		\kappa' (\psi(x) - \psi_*) \cdot S(x)  \geq 1  \quad   \forall x\!:  \mathrm{dist}(x, \Omega) \leq \delta, \;  \psi_* < \psi(x) < \psi_* + \epsilon,  
	\end{equation*}

\vspace{-0.1cm}

\noindent 	where $\kappa: [0,\epsilon] \to \mathbb{R}$ is a concave differentiable function satisfying $\kappa(0) = 0$ and $\kappa'>0$.   
\end{definition}     

\begin{example}
	\label{expl:kl} 
	The KL property  holds for a large class of functions including semi-algebraic functions (e.g., real polynomial functions), vector or matrix (semi)norms (e.g., $\|\cdot\|_p$, with $p \geq 0$ rational number), logarithm functions,  exponential functions and  uniformly convex functions,  see \cite{BolDan:07} for a comprehensive list. 
\end{example}

\noindent If $\psi$ is semi-algebraic,  there  exist $r >1$ and $\sigma_r>0$ such that  $ \kappa$ in Definition \ref{def:kl} is of the form  $\kappa (t) = \sigma_r^{\frac{1}{r}} \frac{r}{r-1} t^{\frac{r-1}{r}}$ \cite{BolDan:07}. In this case  the uniformized  KL property establishes the following local geometry of the nonconvex function $\psi$ around a compact set~$\Omega$:
\begin{equation}
	\label{eq:kl}
	\psi(x) - \psi_*  \leq \sigma_r  S(x)^r \quad   \forall x\!: \;  \text{dist}(x, \Omega) \leq \delta, \; \psi_* < \psi(x) < \psi_* + \epsilon.  
\end{equation}

\noindent If   $\psi$ is uniformly convex of degree $q\geq 2$,  then  from \eqref{eq:probConv} we get that $ \kappa$ in Definition \ref{def:kl} is of the form  $\kappa (t) = (q/\sigma_q)^{\frac{1}{q}}  (q-1)^{\frac{q-1}{q}} t^{\frac{1}{q}}$. Note that the relevant aspect of the KL property is when $\Omega$ is a subset of \textit{critical points} for $\psi$, i.e.,  $\Omega \subseteq \{x: 0 \in \partial \psi (x) \}$, since it is easy to establish the KL property when $\Omega$ is not related to critical points.

\noindent Finally, for a function $\psi$, we denote its sublevel set at a given $x_0 $ by:
$$
\mathcal{L}_\psi(x_0) = \{x \in \Eb:\; \psi(x) \leq \psi(x_0)\}.
$$


\section{General higher-order majorization-minimization algorithms}
\label{sect:ghom}
In what follows, we study the following general nonsmooth  optimization problem:
\begin{equation} 
	\label{eq:optpb}
	\min_{x \in \mathcal{X}} f(x), 
\end{equation}
where $f:  \Eb \rightarrow \bar{\mathbb{R}}$ is a proper (possibly nonconvex) lower semicontinuous (lsc)  function  defined over   its domain   $\text{dom} f $ and $\mathcal{X} \subseteq \text{dom} f$ is a  closed convex set having nonempty  relative interior (this condition ensures validity of  the calculus rules for subdifferentiability of the function $f + \mathbf{1}_{\mathcal{X}}$).   Since $f$ is extended valued, it allows the inclusion of other constraints in addition to  $x \in \mathcal{X}$. We assume that a solution $x^* \in \mathcal{X} $ exists for problem \eqref{eq:optpb}, hence  the optimal value is finite and $f$ is bounded from below by some  $f^* > -\infty$.  In the convex case we consider $f^* = f(x^*)$ (the optimal value). Throughout the paper we assume that $f$ admits directional derivatives and the limiting subdifferential is nonempty at any point in $\mathcal{X}$, respectively. Then,  at a regular point of (local) minimum  $x^* \in \mathcal{X} $,  the first-order necessary optimality conditions for \eqref{eq:optpb} can be written as follows \cite{Roc:70}[Theorem 8.15]:
\begin{align} 
	\label{eq:optcond}
	Df(x^*)[x-x^*] \geq 0 \quad \forall x \in  \mathcal{X}   \quad \iff \quad 0 \in \partial f(x^*) + \mathcal{N}_{\mathcal{X}}(x^*).  
\end{align}
Moreover, when $\text{dom} \,  f =\mathcal{X}  = \Eb$ these conditions are equivalent to $0 \in \partial f(x^*)$.  When $f$ is convex function,  the conditions 	\eqref{eq:optcond}  are also  sufficient for $x^*$ to be a global minimum of function $f$ over the convex set  $\mathcal{X}$.  The  first-order optimality conditions  	\eqref{eq:optcond} are convenient for defining an approximate solution to problem 	\eqref{eq:optpb}.  A point  $\hat x \in \mathcal{X}$  satisfies the first-order optimality conditions of (local) minimum of function $f$ over the convex  set $\mathcal{X}$  with accuracy $\epsilon$ if: 
\begin{align*} 
\inf_{x \in  \mathcal{X}}	\frac{Df(\hat x)[x-\hat x]}{\|x - \hat x\|} \geq - \epsilon.  
\end{align*}
In our convergence analysis below we assume that the sublevel set  $\mathcal{L}_f(x_0)$ is bounded. Then, there exists $R > 0$ such that:
$$
\|x-x^* \| \leq R \qquad \forall x\in \mathcal{L}_f(x_0).
$$

\noindent These basic  assumptions are standard in the literature, see e.g.,  \cite{CarGou:11, cartis:2020, Nes:19, BriGar:17, NesPol:06}. The main approach we adopt in solving the optimization problem \eqref{eq:optpb} is to use a class of functions that approximates well the objective function $f$ but are easier to minimize. We call this class  \textit{higher-order surrogate}  functions. The main properties of our higher-order surrogate function are described  next: 

\begin{assumption}
	\label{def:sur}
For a	given  proper lower semicontinuous (lsc) function $f:  \Eb \rightarrow \bar{\mathbb{R}}$ and  a  closed convex set $ \mathcal{X} \subseteq \text{dom} f $, there exist  an integer $p \geq 1$  and a proper lsc function  $g(\cdot;x): \Eb  \rightarrow \bar{\mathbb{R}}$ at any $x \in \mathcal{X}$, called \textit{$p$ higher-order surrogate of $f$ at $x$},  having  $\text{dom} \, g(\cdot;x) =  \text{dom} f $ and  satisfying the  following properties:
	\begin{enumerate}
		\item[(i)] the surrogate function is bounded from below by  the original function 
		$$g(y;x) \geq f(y) \quad \forall y \in \mathcal{X}. $$
		\item[(ii)]  \red{ there exists an \textit{error function} $h(\cdot;x): \Eb  \rightarrow \bar{\mathbb{R}}$ with $\text{dom} \, h(\cdot;x)$ an open set,  $p$ times differentiable  and  having the $p$ derivative smooth (i.e., Lipschitz continuous) with Lipschitz constant~$L_p^h$ on the  set   $\mathcal{X} \subseteq \text{dom}\, h(\cdot;x)$,  whose restriction to the set $\mathcal{X}$ agrees with $g-f$ (i.e.,  $h(y;x)= g(y;x) - f(y)$ for all  $y \in \mathcal{X}$). }
		\item[(iii)] the $i$th directional derivatives  of the error function $h$ at any $x \in \mathcal{X}$ satisfy 
		$$D^i h(x;x)[\cdot] = 0 \quad \forall  i =0:p,$$ 
		where $i=0$ means that $h(x;x) = 0$, or equivalently $g(x;x) = f(x)$.
	\end{enumerate}
\end{assumption}

\noindent Note that \cite{Mai:15} provided  similar conditions as in Assumption 	\ref{def:sur}, but only for a first-order surrogate function and used it in the context of stochastic optimization. Next, we give several nontrivial examples of higher-order  surrogate functions. 
\begin{example}
\label{expl:8} 
\textit{(proximal functions).} For a general (possibly nonsmooth and nonconvex) function $f: \mathbb{E} \to \bar{\rset}$ and  a  closed convex set $\mathcal{X} \subseteq \text{dom} f $, one can consider  for any $M_p >0$ and integer $p \geq 1$, the following $p$ higher-order surrogate function:
\vspace*{-0.2cm}
$$ g(y;x) = f(y) + \frac{M_p}{(p+1)!} \| y-x \|^{p+1}   \quad \forall x,y \in \mathcal{X}.$$

\noindent This surrogate satisfies all the conditions of Assumption	\ref{def:sur}. Indeed,   $\text{dom} \, g(\cdot;x) =  \text{dom} f $ and  since $M_p\geq 0$, the first property of the surrogate  is immediate. Moreover,  the error function $h(y;x) =  \frac{M_p}{(p+1)!} \| y-x \|^{p+1}$ has $\text{dom} \, h(\cdot;x) = \Eb$ and satisfies $h=g-f$ on any set $\mathcal{X} \subseteq  \text{dom}\, f  \subseteq  \text{dom}\, h(\cdot;x)$.  Next, according to Example \ref{expl:1}, the error function, $ h(y;x) = \frac{M_p}{(p+1)!} \| y-x \|^{p+1}$, has the  $p$ derivative Lipschitz continuous with constant $L_{p}^{h}=M_p$. For the last property of a surrogate, we notice that the $i$ derivative $\nabla^i( \| y-x \|^{p+1})_{|_{\substack{y=x}}} = 0$  for all $i =0:p$. Thus,  $D^i h(x;x) [\cdot]= 0 \;\;  \forall i =0:p$.  Note that  the  function $h$ is always convex w.r.t. $y$, while the surrogate function $g$ is convex only when the objective function $f$ is convex.   
\end{example}


\begin{example}
	\label{expl:88} 
	\textit{(hessian aware proximal functions).}  \red{For a  nonconvex function $f: \mathbb{E} \to \rset$ that is twice differentiable  and  with the second derivative Lipschitz continous with constant $L_2^f$ on the full domain $\mathbb{E}$  and a  closed convex set $\mathcal{X} \subseteq \mathbb{E}$, one can consider  for any $ M_2 \geq  L_2^f$ the following second-order surrogate function:
\begin{align}
	\label{sur:pp}
	 g(y;x) =f(y) -   \frac{1}{2}  (y-x)^T \nabla^2_{-}f(x)(y-x) +  \frac{M_2}{3} \| y-x \|^{3}   \quad \forall x,y \in \mathcal{X}, 
\end{align}
where $\nabla^2_{-}f(x)$ denotes the negative part of the Hessian of $f$ at $x$\footnote{For a symmetric matrix $A$, with both positive and negative eigenvalues, the $LDU$ factorization leads to $U=L^T$  and $D$ diagonal which can be separated into the positive portion $D_+$ and the negative portion $D_{-}$ (they have all positive or all negative eigenvalues and zeros). Then, $A = L (D_+ + D_{-}) L^T$ and the negative part is $A_{-} = L  D_{-} L^T $.}. First,  one should notice that this surrogate is always a convex function in  $y$, even if  $f$  non-convex. Indeed, since  $f$  has the Hessian Lipschitz continuous, from \eqref{eq:TayAppG2} we have:
\begin{align*}
\nabla^2_y g(y;x) &= \nabla^2 f(y) - \nabla^2_{-} f(x) + M_2\left( \|y-x\| B + \frac{B(y-x)(y-x)^T B}{\|y-x\|}\right) \\
	& \succeq  \nabla^2 f(y) - \nabla^2 f(x) + M_2\left( \|y-x\| B + \frac{B(y-x)(y-x)^T B}{\|y-x\|}\right) \\
	&\overset{\eqref{eq:TayAppG2}}{\succeq} -L_2^f \|y-x\|B + M\left( \|y-x\| B + \frac{B(y-x)(y-x)^T B}{\|y-x\|}\right) \\
	&= (M_2-L_2^f) \|y-x\| B + M_2 \frac{B(y-x)(y-x)^T B}{\|y-x\|} \succeq 0 \quad \forall x, y \in \mathbb{E}.
\end{align*} 
Hence, compared to the surrogate of Example \ref{expl:8}, the second-order surrogate  	\eqref{sur:pp} is always convex  for any fixed $ x \in \mathbb{E}$ (although $f$ is nonconvex). In fact,  one can easily notice that this surrogate is uniformly convex  (if $M_2> L_2^f$) and has the Hessian Lipschitz continuous. Hence, it can be minimized  by a   gradient type method that has a linear rate of convergence depending only on absolute constants, see  e.g.,  \cite{Nes:19,Nes:19Inxt}. }
 
\medskip 

\noindent \red{ Moreover, this surrogate 	\eqref{sur:pp} satisfies all the conditions of Assumption 	\ref{def:sur}.  Indeed, we have $\text{dom} \, g(\cdot;x) =  \text{dom} f = \mathbb{E}$ and since for any given $x \in \mathbb{E}$ we have  $ - \nabla^2_{-} f(x) \succeq 0$ and $M_2 \geq 0$, it follows that $g(y;x) \geq f(y)$ for any $y \in \mathbb{E}$, i.e., the first property of the surrogate holds. Moreover,   the error function $h(y;x) = -   \frac{1}{2}  (y-x)^T \nabla^2_{-}f(x)(y-x) +  \frac{M_2}{3} \| y-x \|^{3} $ has $\text{dom} \, h(\cdot;x) = \mathbb{E} $,  satisfies  $h=g-f$  on any set $\mathcal{X} \subseteq \Eb$  (including  $\mathcal{X}  = \mathbb{E}$) and   is always convex w.r.t. $y$. Additonally,   this surrogate has the gradient Lipschitz continuous  with Lipschitz constant $L_1^h = \max_{x,y \in \mathcal{L}_f(x_0) \cap \mathcal{X}}  \|   - \nabla^2_{-} f(x) + M_2( \|y-x\| B + B(y-x)(y-x)^T B \|y-x\|^{-1})  \|$ (recall that the level set $\mathcal{L}_f(x_0)$ is assumed bounded).  Finally, we have  $g(x;x) = f(x)$ and $\nabla g(x;x) = \nabla f(x)$ for all $x \in \mathbb{E}$. Hence, according to our GHOM algorithm defined below, this surrogate  provides one of the first  implementable higher-order proximal point type algorithms for  minimizing twice differentiable  smooth nonconvex functions that enjoys convergence guarantees. }

\vspace{-0.1cm}

\noindent  \red{ For a convex function $f$ we can change the previous surrogate as follows:
$$ g(y;x) \!=\! f(y) -   \frac{1}{2}  (y-x)^T \nabla^2f(x)(y-x) +  \frac{M_0}{2} \| y-x \|^{2}  +  \frac{M_2}{3} \| y-x \|^{3}   \;\; \forall x,y  \!\in\!  \mathcal{X}, $$
 where  $M_2 \geq  L_2^f$ and $M_0 \geq 0$. Following the same analysis as before, one can easily show that  this surrogate is  convex (even for $M_0=0$) and satisfies all the conditions of Assumption 	\ref{def:sur} (however, in order to ensure $g(y;x) \geq f(y)$ we need $M_0 \geq \| \nabla^2f(x)\| $).  Of course, cubic Newton is  an efficient  scheme for minimizing  functions with second derivative smooth. However, GHOM algorithm based on the previous surrogates leads to  proximal point type schemes which usually yields larger steps (i.e., larger decrease in the objective) than cubic Newton scheme. For example, if $T_\text{CN}$ is the cubic Newton step   and $T_\text{PP}$ is the proximal point step from $x$, we have:
 \begin{align*}
 & f(T_\text{CN})  \leq f(x) - \frac{\bar M_2}{12}\|  T_\text{CN} - x\|^3  \quad (\text{see} \;  [22])	\\
& f(T_\text{PP}) \leq f(x) - \frac{ M_2}{3}\|  T_\text{PP} - x\|^3 - \frac{1}{2}  (T_\text{PP} -x)^T  \left( M_0 B -  \nabla^2f(x)\right) 	(T_\text{PP} -x),
 \end{align*}	
respectively, where $\bar M_2, M_2 \geq  L_2^f$. The last inequality follows from $g(T_\text{PP};x) \leq g(x;x) = f(x) $ and note that $M_0 B -  \nabla^2f(x)\succeq 0$ if $M_0 \geq \| \nabla^2f(x)\| $. A similar decrease  is also valid  for the surrogate  	\eqref{sur:pp}.   
}
\end{example}	


\begin{example}
	\label{expl:4}
	\textit{(smooth derivative functions).}  For a  function $f: \mathbb{E} \to \rset$ that is $p \geq 1$ times differentiable  and  with the $p$ derivative Lipschitz on the full domain $\mathbb{E}$ of constant $L_p^f$ and a  closed convex set $\mathcal{X} \subseteq \mathbb{E}$, one can consider  for any $ M_p \geq  L_p^f$ the following $p$ higher-order surrogate function over $\mathcal{X}$:
	$$ g(y;x) =T_p^f(y;x) + \frac{M_p}{(p+1)!} \| y-x \|^{p+1}   \quad \forall x,y \in \mathcal{X}. $$
	
	\vspace{-0.2cm}

\noindent  This higher-order surrogate also satisfies all the conditions of Assumption 	\ref{def:sur}.  Indeed, since  $ M_p \geq  L_p^f$, from \eqref{eq:TayAppBound} we have $g(y;x) \geq f(y)$ for any $y$,  i.e., the first property of the surrogate holds. Moreover, the error function $h(y;x)=g(y;x)-f(y)$ has $\text{dom} \, h(\cdot;x)= \text{dom} \, g(\cdot;x) =  \text{dom} f  =  \mathbb{E}$ and $\nabla^p h(y;x) = \nabla^p f(x) - \nabla^p f(y) + \nabla^p(M_p/(p+1)! \, \|y-x\|^{p+1}) $. Hence, for $x$ fixed,  $h$ is obviously  $p$ smooth with  Lipschitz constant $L_p^h =  L_p^f + M_p $ and satisfies also the third property of a surrogate.  Note that if we choose  $M_p \geq  p L_p^f$, then from  \eqref{eq:TayAppG2} it follows that    the error function  $h$ is always  convex, while  the  higher-order surrogate  $g$  is convex  only if $f$ is convex  (see Lemma \ref{lema:conv}).  Note that in this example we can  consider any  convex set $\mathcal{X}  \subseteq  \mathbb{E}$.  
\end{example}


\begin{remark} 
	\label{rmk:01}
	If the function  $f$ additionally satisfies $f(y) \geq T_{p}^{f}(y;x)$ for all $y$, then we get for $h$ a better bound for the constant in the right hand side of \eqref{eq:TayAppBound} than $L_p^h = M_p + L_p^f$.  Indeed, using the third property of a  $p$ higher-order surrogate function, we have $T_p^h(y;x) =0$ for all $y$ and thus:
	\vspace*{-0.3cm}
	\begin{align*}
		|h(y;x) - T_p^h(y;x) | &\!= |g(y;x) - f(y)| =       
		  \left|f(y) - T_p^f(y;x) - \frac{M_p}{(p+1)!} \|y-x\|^{p+1} \right|.     
	\end{align*}

\vspace{-0.3cm}

\noindent We observe that the last term in the equation above is positive and due to the additional condition $f(y) \geq T_p^f(y;x)$, we also have that $f(y)- T_p^f(y)\geq 0$. Hence, using  $|a-b | \leq \max\{a,\, b\}$ for any two positive scalars $a$ and $b$,  we further  get:
		\vspace*{-0.2cm}	
	\begin{align*}
		& |h(y;x) - T_p^h(y;x) |  \leq \max \left(f(y) - T_p^f(y;x), \,\frac{M_p}{(p+1)!} \|y-x\|^{p+1} \right) \\
		& \overset{\eqref{eq:TayAppBound}}{\leq}   \max \left( \frac{L_p^f}{(p+1)!} \| y-x\|^{p+1},    \,\frac{M_p}{(p+1)!} \|y-x\|^{p+1} \right) =  \frac{M_p}{(p+1)!} \|y-x\|^{p+1}.
	\end{align*}

\vspace{-0.2cm}

\noindent Therefore, in this case  we can use in \eqref{eq:TayAppBound} for function $h$ constant   $M_p$ instead of $L_p^h = M_p + L_p^f$.  Functions that satisfy the additional condition from Remark \ref{rmk:01}, are e.g., the convex functions  for $p=1$ (since in this case from convexity of $f$ we automatically have $f(y) \geq T_{1}^{f}(y;x)$).  
\end{remark}


\begin{example}
	\label{expl:5}
	\textit{(composite functions).} Let   $f_1: \Eb \to \rset$ have the  $p \geq 1$ derivative smooth on $\Eb$ with constant $L_p^{f_1}$, $f_2: \Eb \rightarrow \bar{\mathbb{R}}$ be a proper closed convex function with domain $\text{dom} f_2$ and $\mathcal{X} \subseteq \text{dom} f_2$ a closed convex set. Then, for the composite function  $f= f_1 +f_2$ one can take the following  $p$ higher-order surrogate:
	$$
	g(y;x) = T_p^{f_1}(y;x) +\frac{M_p}{(p+1)!} \|y -x\|^{p+1} +  f_2(y) \quad \forall x,y \in \mathcal{X},
	$$
\noindent 	where $M_p \geq  L_p^{f_1}$. This surrogate  satisfies all the conditions of Assumption 	\ref{def:sur}. Indeed, we observe that   $\mathcal{X} \subseteq \text{dom}f  = \text{dom} \, g(\cdot;x) = \text{dom}f_2$ and   the  error function $h(y;x)= T_p^{f_1}(y;x) +\frac{M_p}{(p+1)!} \|y -x\|^{p+1} -  f_1(y)$   has $\text{dom} \, h (\cdot;x)= \Eb$ (thus $\mathcal{X} \subseteq \text{dom}\, h(\cdot;x)$). Following the same arguments  as in Example  	\ref{expl:4} we obtain that  $h$ is  $p$ smooth with  Lipschitz constant $L_p^h =L_p^{f_1} + M_p$.  Additionally,  if we choose  $M_p \geq  p L_p^{f_1}$, then   the error function  $h$ is always  convex, while  the  higher-order surrogate function $g$  is convex  only if $f$ is convex.
\end{example}


\begin{example}
	\label{expl:888}
	\red{  \textit{(difference of convex functions).}   Let  the function $f_1: \Eb \rightarrow \bar{\mathbb{R}}$ be  proper closed and convex with domain $\text{dom} f_1$,  the function $f_2: \Eb \to \rset$ have the  $p_2\geq 1$ derivative smooth on $\Eb$ with constants $L_{p_2}^{f_2}$ and a convex set $\mathcal{X} \subseteq \text{dom} f_1$. Then, for the difference of convex functions  $f= f_1 - f_2$ one can take the following   higher-order surrogate having two diferent reguralization powers:
		$$
		g(y;x) = f_1(y) +\frac{M_{p_1}}{(p_1+1)!} \|y -x\|^{p_1+1} \!- T_{p_2}^{f_2}(y;x) +\frac{M_{p_2}}{(p_2+1)!} \|y -x\|^{p_2+1} \;\; \forall x,y \!\in\! \mathcal{X},
		$$
		\noindent 	where $M_{p_1} \geq  0,  M_{p_2} \geq  L_{p_2}^{f_2}$ and  $p_1$ satisfies $ 1 \leq p_1 \leq p_2$ and it is odd.  Using the same arguments as for the previous examples, one can easily see that this surrogate also satisfies all the conditions of Assumption	\ref{def:sur} for $p=p_2$.  In particular, one can notice that the error function $h(y;x) = \frac{M_{p_1}}{(p_1+1)!} \|y -x\|^{p_1+1} \!- T_{p_2}^{f_2}(y;x) +\frac{M_{p_2}}{(p_2+1)!} \|y -x\|^{p_2+1} + f_2(y)$ is  $p_2$ smooth as $\nabla^{p_2}(M_{p_1}/(p_1+1)! \, \|y-x\|^{p_1+1})$ is either constant  or linear, provided that $p_1$ is odd and $p_1 \leq p_2$.  Note that the surrogate corresponding to  $p_1=p_2=1$ is frequently employed in the  difference of convex functions literature. Moreover, if  $f_1$ is e.g., a simple function or a  convex quadratic function and $p_2=2$ one can easily minimize this surrogate  using  the same  techniques developed for cubic Newton method \cite{NesPol:06}. Hence, GHOM algorithm based on such surrogate is one of the first (implementable) schemes for minimizing difference of convex functions that  uses higher  information than of gradient type.    Other choices for the higher-order  surrogate are possible depending on the smoothness  of the functions $f_1$ and/or $f_2$.}
\end{example}


\begin{example}
	\label{expl:6}
	\textit{(bounded  derivative functions).} 
	We consider a function $f: \Eb \rightarrow \mathbb{R}$ that is $p+1$ times continuously differentiable  and  a closed convex set $\mathcal{X} \subseteq \mathbb{E}$. Assume that the $p+1$ derivative of $f$ is upper bounded by a positive constant   $M_{p+1}$ on $\mathcal{X}$,  i.e.,  $\|  D^{p+1}f(x)\| \leq M_{p+1} \; \forall x \in \mathcal{X}$, or, if $p$ is  odd,  bounded by two $p+1$ symmetric multilinear forms $\mathcal{Q}_{p+1}, \mathcal{H}_{p+1}$ on $\mathcal{X}$, i.e.,  $\mathcal{Q}_{p+1}[h]^{p+1} \leq  D^{p+1}f(x)[h]^{p+1} \leq \mathcal{H}_{p+1}[h]^{p+1} \; \forall x \in \mathcal{X}, h \in \Eb$ (e.g., for $p=1$ this means that the Hessian is bounded on $\mathcal{X}$ by two symmetric matrices, not necessarily positive semidefinite, i.e., $f$ is a weakly convex function with smooth gradient).  Then, the $p$ derivative of $f$  is Lipschitz  continuous on $\mathcal{X}$ with $L_p^f = M_{p+1}$ or  $L_p^f = \max(\|\mathcal{H}_{p+1}\|, \|\mathcal{Q}_{p+1}\|)$, and then one has the following  $p$ higher-order surrogates:
	\vspace*{-0.2cm}
	$$
	g(y;x) = T_p^f(y;x) + \frac{\bar M_{p+1}}{(p+1)!} \|y-x\|^{p+1}  \; \text{or} \; =T_p^f(y;x) + \frac{ \mathcal{H}_{p+1}[y-x]^{p+1}}{(p+1)!}  \;  \forall x,y \in \mathcal{X},
	$$ 
	
	\vspace{-0.2cm}
	
\noindent where $\bar M_{p+1} \geq  M_{p+1}$.  One can easily see that these surrogates satisfy  Assumption 	\ref{def:sur}. The first property of the surrogate  follows  from Taylor's theorem  with Lagrange form of the remainder.  Moreover, the error function $h= g-f$ has $\text{dom} \, h(\cdot;x) = \text{dom} \, f = \mathbb{E}$ and is $p$ smooth on $\mathcal{X}$  with  Lipschitz constant   $L_p^h =  M_{p+1} + \bar M_{p+1}$ or  $L_p^h =  L_p^f + \|\mathcal{H}_{p+1}\|$.  \red{Note that the  functions considered  in  Example  \ref{expl:4} may be different from those covered here,  as one can find functions with the $p$ derivative Lipschitz continuous, but the $p+1$ derivative may not exist everywhere.}  Moreover, our algorithm GHOM defined below  based on the previous surrogates  with $p=1$ is  a  (projected) gradient or Netwon scheme  with constant Hessian.  
\end{example}

\noindent The reader can find other examples of  higher-order surrogate functions depending on the structure of the objective function $f$ in  \eqref{eq:optpb} and we believe that this paper opens a window of opportunity for higher-order algorithmic research.  

\medskip 

\noindent Now, we are ready to define our General Higher-Order Majorization-Minimization (GHOM) algorithmic framework.

\begin{algorithm}
	\caption{Algorithm  GHOM}
	\label{alg:buildtree}
	\begin{algorithmic}
		\STATE{Given $x_0 \in  \mathcal{X}$ and $p \geq 1$,  for $k\geq 0$ do:}
		\STATE{	1. Compute the $p$ surrogate function $g(y;x_k)$ of $f$ near $x_k$}
		\STATE{	2. Compute a stationary point $x_{k+1}$ of the subproblem: 
			\begin{equation}
				\label{eq:sp}
				\min_{y \in \mathcal{X}} g(y;x_k),
			\end{equation} 
			satisfying the  following descent property
			\begin{equation}
				\label{eq:descGhom}
				g(x_{k+1};x_k) \leq g(x_k;x_k) \quad \forall k \geq 0.
			\end{equation} 	
		}
	\end{algorithmic}
\end{algorithm}

\medskip 
\noindent   From Examples \ref{expl:8}, \ref{expl:88}, \ref{expl:4}, \ref{expl:5} and    \ref{expl:6}, one can notice that the algorithm GHOM is  simple and depends on at most one parameter, $M_p$, whose value can be determined through a line search procedure when e.g., the Lipschitz constant is unknown.  Moreover, GHOM generates $x_k \in \mathcal{X}$ for all $k \geq 0$.  If $f$ is convex function, by choosing the regularization parameter of the surrogate appropriately  we can also have $g(\cdot;x_k)$ convex function (see Lemma \ref{lema:conv}) and thus we can use efficient methods from  convex optimization to find the global solution $x_{k+1}$ of the convex subproblem  \eqref{eq:sp} at  each iteration, see e.g.,   \cite{CarGou:11, CarDuc:16, GraNes:19sub, Nes:19Inxt}.  When $f$ is nonconvex function, our convergence analysis below requires only the computation of a stationary point $x_{k+1}$ for the subproblem \eqref{eq:sp} (i.e., finding $x_{k+1} \in \mathcal{X}$ such that $0 \in \partial g(x_{k+1};x_k) + \mathcal{N}_{\mathcal{X}}(x_{k+1})$ or equivalently $D g(x_{k+1};{x}_{k})[y -x_{k+1}]  \geq 0$ for all $y \in \mathcal{X}$) satisfying, additionally,  the descent  \eqref{eq:descGhom}.  Moreover, when $\text{dom} \,  g =\mathcal{X}  = \Eb$ the previous  stationarity  conditions are reduced to $0 \in \partial g(x_{k+1};x_k)$. Note that the vast majority of nonconvex optimization algorithms are able to identify stationary points of nonconvex problems. 



\section{ Convergence analysis  of GHOM for  convex optimization}
\label{sec:conv}
In this  section we analyze the  convergence of algorithm GHOM under various assumptions on the convexity of the objective function $f$ over the convex set  $\mathcal{X}$.  \red{Note that throughout this section,  where the function $f$ is considered  convex,    we assume that $x_{k+1}$ is  a global minimum of the  convex  subproblem \eqref{eq:sp}. We state this as an  assumption:
\begin{assumption}
\label{ass:opt}	
When the objective function $f$ is convex we assume that the higher-order surrogate $g(\cdot;x)$ is also a convex function for any $x \in \mathcal{X}$ and $x_{k+1}$ is  a global minimum of the  convex  subproblem \eqref{eq:sp}.
\end{assumption}		
}

\subsection{Global sublinear convergence  of GHOM}
\label{sec:conv-conv}
In this section we derive  the global rate of convergence of order $\mathcal{O}(1/k^p)$ for  GHOM in terms of function values when \eqref{eq:optpb} is a general convex  optimization problem.  Let  $x^* \in X$ be a  minimum point of this convex problem  \eqref{eq:optpb}.

\begin{theorem}
	\label{th:conv-conv}
Let the  objective function $f$ in the optimization problem \eqref{eq:optpb} be proper, lower semicontinuous and convex on  the closed convex set $\mathcal{X}$. \red{ Additionally,   let Assumptions \ref{def:sur}  and  \ref{ass:opt}	hold. }  Then,  the sequence $(x_k)_{k \geq 0}$  generated by algorithm GHOM has the following global sublinear convergence rate:
	\begin{align}
		\label{theq:conv-conv}
		f(x_{k}) - f(x^*) \leq \frac{L_p^h R^{p+1}}{p! \left(1+\frac{k}{p+1}\right)^{p}}.
	\end{align}
\end{theorem}

\begin{proof} Since  $h(y,x)$, defined as the  error between the $p$ higher-order surrogate $g(y;x)$ and  the function $f$, has the $p$ derivative  smooth with Lipschitz constant  $L_p^h$ , then from  \eqref{eq:TayAppBound} we get:
	$$
	h(y;x) \leq T_p^h(y;x) + \frac{L_p^h}{(p+1)!} \| y-x\|^{p+1} \quad \forall x,y \in \mathcal{X}.
	$$
	However, based on the condition (iii) from Assumption  \ref{def:sur} the Taylor approximation of $h$ of order $p$ at $x$ satisfies:
	$$
	T_p^h(y;x) = h(x;x)+\sum_{i=1}^{p} \frac{1}{i !} D^{i} h(x;x)[y-x]^{i} = 0.
	$$
	\noindent Therefore,  we obtain:
	\begin{align*} 
		h(y;x) = g(y;x) -f(y) &\leq   \frac{L_p^h}{(p+1)!} \| y-x\|^{p+1},
	\end{align*}
	which implies that
	\begin{align}\label{eq:5}
		g(y;x)  & \leq f(y) + \frac{L_p^h}{(p+1)!} \| y-x\|^{p+1} \quad \forall x,y \in \mathcal{X}.
	\end{align}
	Recall that  when $f$ is convex, we consider  $x_{k+1}$  to be a global minimum of $g(\cdot;x_k)$ over convex set   $\mathcal{X} \subseteq \text{dom} \, g(\cdot;x_k) =  \text{dom} f $ and since   $g(\cdot;x_k)$ is bounded from below by $f$ (see Assumption  \ref{def:sur} (i)),  we further get:
	\begin{align}
		\label{eq:main}
		f(x_{k+1}) & \leq g(x_{k+1};x_k) \!= \min_{y \in \mathcal{X}} g(y;x_k) \overset{\eqref{eq:5}}{\leq} \min_{y \in \mathcal{X}} (f(y) + \frac{L_p^h}{(p+1)!} \| y-x_k\|^{p+1}). 
	\end{align}		
	Since $f$ is assumed to be convex,  we can choose $ y = x_k + \alpha(x^* -x_k)$, with $\alpha \in [0,\, 1]$, and obtain further:
	\vspace*{-0.2cm}	
	\begin{align}
		 f(x_{k+1}) & \leq \min_{\alpha\in [0,\, 1]} f(x_k + \alpha(x^* -x_k)) + \frac{L_p^h}{(p+1)!} \| x_k + \alpha(x^* -x_k)-x_k\|^{p+1} \nonumber \\
		& \leq \min_{\alpha\in [0,\, 1]} f(x_k) - \alpha(f(x_k) - f(x^*)) + \frac{L_p^h}{(p+1)!}  \alpha^{p+1}\| x_k-x^*\|^{p+1}.  \label{eq:main0}
	\end{align}
	Let us  show that GHOM is a decent algorithm. Indeed, from \eqref{eq:descGhom} we have:
	\begin{equation*}
		f(x_k) = g(x_k;x_k) \geq   g(x_{k+1};x_{k}) \geq f(x_{k+1})\quad \forall k \geq 0.
	\end{equation*}
	Hence, all the iterates $x_k$ are in the level set $\mathcal{L}_f(x_0)$ and thus satisfy $\|x_k-x^* \| \leq R$ for all $k \geq 0$.  Subtracting the optimal value on both side of  \eqref{eq:main0} and recalling the fact that the sublevel set of $f$ at $x_0$ is assumed bounded, we obtain:
		\vspace*{-0.2cm}
	\begin{equation} \label{eq:6}
		f(x_{k+1}) -f(x^*) \leq \min_{\alpha\in [0,\, 1]} (1 - \alpha)(f(x_k) - f(x^*)) + \frac{L_p^h  R^{p+1}}{(p+1)!} \alpha^{p+1}. 
	\end{equation}
 For simplicity, we denote $\Delta_{k} = f(x_{k}) -f(x^*)$. We consider two cases in \eqref{eq:6} (see also \cite{Nes:19, Nes:19Inxt}):\\
	First case: if $\Delta_k > \frac{L_p^h R^{p+1}}{p!}$, then  optimal point   is $\alpha^*=1$ and  $\Delta_{k+1} \leq \frac{L_p^h  R^{p+1}}{(p+1)!}.$\\
	Second case: if $\Delta_k \leq \frac{L_p^h R^{p+1}}{p!}$, then optimal point is $\alpha^* = \sqrt[p]{\frac{\Delta_kp!}{R^{p+1} L_p^h}}$ and  obtain:
	\begin{equation*}
		\begin{aligned}
			\Delta_{k+1} &\leq \Delta_{k} \left(1-c \Delta_k^{\frac{1}{p}} \right), \quad 
			\Delta_{k+1}^{-\frac{1}{p}} \geq \Delta_k^{-\frac{1}{p}} \left( 1 - c \Delta_k^{\frac{1}{p}} \right)^{-\frac{1}{p}}, \\	
			\Delta_{k+1}^{-\frac{1}{p}} &\geq \Delta_k^{-\frac{1}{p}} \left( 1+ \frac{c}{p} \Delta_k^{\frac{1}{p}} \right) = \Delta_k^{-\frac{1}{p}} + \frac{c}{p},\\ 
		\end{aligned}
	\end{equation*} 

	\vspace{-0.2cm}

\noindent 	where $c = \frac{p}{p+1} \sqrt[p]{\frac{p!}{R^{p+1} L_p^h}}$ and the last inequality is given by $(1-x)^{-p} \geq 1+ px $ for $x \in [0,\, 1]$, see e.g.,  \cite{Pol:87}.  We now apply recursively the previous inequalities, starting with $k = 1$. If $\Delta_0\geq \frac{L_p^h R^{p+1}}{p!}$, we are in the first case and then $\Delta_1 \leq \frac{L_p^h  R^{p+1}}{(p+1)!}$. Then, we will subsequently be in the second case for all $k \geq 2$ and  
	$$\Delta_k \leq \Delta_1 \left(1+ \frac{(k-1)c}{p} \Delta_1^{\frac{1}{p}} \right)^{-p} \leq \frac{L_p^h R^{p+1}}{(p+1)!} \left( 1+\frac{k-1}{p+1} (p+1)^{\frac{1}{p}} \right)^{-p}. $$ Otherwise, if $\Delta_0\leq \frac{L_p^h R^{p+1}}{p!}$, then we are in the second case and obtain:
	$$\Delta_k \leq \Delta_0 \left( 1+ \frac{kc}{p} \Delta_0^{\frac{1}{p}} \right)^{-p} \leq \frac{L_p^h R^{p+1}}{p!} \left( 1+\frac{k}{p+1} \right)^{-p}. $$
	These prove the statement of the theorem.  
\end{proof}

\noindent  Note that  in the previous proof we have not used that the surrogate $g(\cdot;x_k)$ is a  convex function, we only used that $x_{k+1}$ is  a global minimum of the   subproblem \eqref{eq:sp}. Moreover,  for $\mathcal{X} = \Eb$ our convergence rate \eqref{theq:conv-conv} recovers the usual convergence rates $\mathcal{O}(1/k^p)$ of higher-order Taylor-based methods in the unconstrained convex  case \cite{Nes:19} (Example \ref{expl:4}) and composite convex  case \cite{Nes:19Inxt} (Example \ref{expl:5}), respectively.  However, from our best knowledge, other examples of surrogate functions (such as Examples \ref{expl:88}, \ref{expl:6}) have not been investigated in the literature yet.   Moreover, the  convergence  results from \cite{Nes:19, Nes:19Inxt} assume Lipschitz continuity of the $p$ derivative of the objective function $f$, while Theorem \ref{th:conv-conv}  assumes Lipschitz continuity of the $p$ derivative of the error function $h=g-f$. Hence, our proof is different from \cite{Nes:19, Nes:19Inxt}.  Therefore, Theorem \ref{th:conv-conv} provides  a unified convergence analysis for higher-order majorization-minimization algorithms, that covers in particular  unconstrained and composite convex  problems, under possibly more general assumptions than in  \cite{Nes:19, Nes:19Inxt}.  In fact there  is a major difference between the Taylor expansion approach from  \cite{Nes:19, Nes:19Inxt} and the model approximation based on our general majorization-minimization approach. Taylor expansion yields  unique approximation model around a given point, while in the majorization-minimization approach one may consider  many upper bound models and every model leads to a different optimization~method.


\subsection{Superlinear  convergence  of GHOM}
\label{subsection3.2}
\noindent Next, by assuming  uniform convexity of $f$ on $\text{dom}f$, we prove that GHOM can achieve faster rates. More precisely we get   local  \textit{superlinear} convergence rates for GHOM in several optimality  criteria:  function values, distance of the iterates to the optimal point  and norm of  subgradients.  For these derivations we first need some auxiliary results.  	

\begin{lemma}\label{lema:2}
Let the assumptions of Theorem \ref{th:conv-conv} hold. Then, for the sequence $(x_{k})_{k\geq 0}$  generated by  algorithm GHOM we have   $-  \nabla h(x_{k+1};x_k) \in \partial (f+\textbf{1}_\mathcal{X})  (x_{k+1}) $ and the following relation holds:
\begin{equation}
	\label{eq:26}
	\| \nabla h(x_{k+1};x_k)\|_* \leq \frac{L_p^h}{p!} \| x_{k+1} -x_k \|^p \quad \forall k \geq 0.
	\end{equation}
\end{lemma}	

\begin{proof} 
Since  $x_{k+1}$  is the global minimum of   the convex subproblem \eqref{eq:sp} (see Assumption \ref{ass:opt}), we have that there exists $g^{x_{k+1}} \in \partial g(x_{k+1};x_k)$ such that $ \langle g^{x_{k+1}} , y - x_{k+1} \rangle \geq 0$ for all $y \in \mathcal{X}$. Since the error function $h(y;x_k) = g(y;x_k) - f(y)$ is differentiable, we obtain from basic calculus rules  for convex functions  that  there  exists a subgradient  $f^{x_{k+1}} \in  \partial f(x_{k+1})$ such that  \cite{Roc:70}:
\[    f^{x_{k+1}} +  \nabla h(x_{k+1};x_k) =   g^{x_{k+1}}.   \]	

\noindent Since $f$ is convex function, we have:
\begin{align*}
	f(y) & \geq f(x_{k+1}) +  \langle f^{x_{k+1}} , y - x_{k+1} \rangle  =   f(x_{k+1}) +  \langle g^{x_{k+1}}  -  \nabla h(x_{k+1};x_k) , y - x_{k+1} \rangle   \\
	& \geq  f(x_{k+1}) +  \langle - \nabla h(x_{k+1};x_k) , y - x_{k+1} \rangle \quad \forall y \in  \mathcal{X}.
\end{align*}		
Since $f+\textbf{1}_\mathcal{X}$ is also convex and $x_{k+1} \in  \mathcal{X}$, from the previous relation we get  that $-  \nabla h(x_{k+1};x_k) \in \partial (f+\textbf{1}_\mathcal{X})  (x_{k+1}) $, which proves the first part of the lemma.  Further,	from the definition of the $p$ higher-order surrogate function, we know that the error function $h$ has the $p$ derivative Lipschitz continuous with constant $L_p^h$. Thus, we have the inequality \eqref{eq:TayAppG1} for $y =x_{k+1}$, that is:
	\begin{equation*}
		\|\nabla h(x_{k+1};x_k)  - \nabla T_p^h(x_{k+1}; x_k) \|_* \leq \frac{L_p^h}{p!} \|x_{k+1} - x_k \|^{p}.
	\end{equation*}
\noindent Since the Taylor approximation of $h$ at $x_k$, $T_p^h(y; x_k)$, is zero {for any $y \in \mathcal{X}$},  we get: 
	\begin{align}
		\label{eq:hfsubg}
		\|\nabla h(x_{k+1};x_k) \|_*  \leq \frac{L_p^h}{p!} \|x_{k+1} - x_k \|^{p},
	\end{align}
which proves the second part of the lemma. 	
\end{proof}


\begin{lemma} \label{lema:1}
Let the assumptions of Theorem \ref{th:conv-conv} hold and additionally assume that the error function $h = g-f$ is convex on $\mathcal{X}$. Then, for the sequence $(x_{k})_{k\geq 0}$ generated by the algorithm GHOM,  the following relation holds:
	\begin{equation}\label{eq:27}
	\langle \nabla h(x_{k+1};x_k),  x_k -x_{k+1} \rangle \leq 0 \quad \forall k \geq 0.
	\end{equation}
\end{lemma}

\begin{proof}
The error function $h$ is $p \geq 2$ differentiable satisfying $h(y;x_k) \geq 0$ and $h(x_k;x_k) = 0$.  Combining these relations with the convexity of $h(\cdot;x_k)$, we have:
	\begin{align*}
		0  = h(x_k;x_k)  & \overset{ h \; \text{convex}}{\geq} h(x_{k+1};x_k) + \langle \nabla h(x_{k+1};x_k),  x_k -x_{k+1} \rangle \\
		& \overset{ h \geq 0}{\geq} \langle \nabla h(x_{k+1};x_k),  x_k -x_{k+1} \rangle,
	\end{align*}
 which yields our statement.
\end{proof}

\noindent Now we are ready to prove the local superlinear convergence of the GHOM algorithm in function values for general uniformly convex objective functions. 

\begin{theorem}
	\label{th:conv-conv_super0}
Let the  objective function $f$ in the optimization problem \eqref{eq:optpb} be proper, lower semicontinuous and  uniformly convex  of degree $q \geq 2$ with constant $\sigma_{q}$  on  the closed convex set $\mathcal{X}$.  Additionally,  \red{let Assumptions \ref{def:sur}  and  \ref{ass:opt}	hold} and assume that  the error function $h = g-f$ is convex on the convex set $\mathcal{X}$.  Then, the sequence $(x_k)_{k \geq 0}$ generated by algorithm GHOM has the following  convergence rate in function values:
	\begin{align}
		\label{theq:conv-conv_super0}
		f\left(x_{k+1}\right) - f(x^*) \leq (q - 1) q^{\frac{p-q+1}{q-1}}  \left(\frac{1}{\sigma_{q}}\right)^{\frac{p+1}{q-1}}   \left(\frac{L_{p}^h}{p !}\right)^{\frac{q}{q-1}}  \left(f\left(x_{k}\right) - f(x^*) \right)^{\frac{p}{q-1}}.
	\end{align}
Additionally, $\|  x_{k+1} - x_k\| $ converges to zero, i.e.:
\begin{align}
\label{deltax}
\|  x_{k+1} - x_k\|  \leq \left( \frac{q}{\sigma_q} (f(x_k) - f(x^*))	\right)^{\frac{1}{q}}. 
\end{align}	 
\end{theorem}

\begin{proof}
From Lemma \ref{lema:2} we have $-  \nabla h(x_{k+1};x_k) \in \partial (f+\textbf{1}_\mathcal{X})  (x_{k+1}) $.  Recall that  $x^* \in \mathcal{X}$ denotes a minimum point of the convex problem 	\eqref{eq:optpb} and $x_k, x_{k+1} \in \mathcal{X}$. Then, using Lemma 	\ref{lema:uc}, 	for any $k \geq 0$  we have:
	\begin{align*}
		f(x_k) - f(x^*) &\geq f(x_{k}) - f(x_{k+1})  \\ & \overset{\text{Lemma}\, \ref{lema:uc}}{\geq} \langle -  \nabla h(x_{k+1};x_k), x_k-x_{k+1}\rangle+\frac{\sigma_{q}}{q}\|x_k-x_{k+1}\|^{q}\\
		& \overset{\eqref{eq:27}}{\geq} \frac{\sigma_{q}}{q}\|x_k-x_{k+1}\|^{q} \overset{\eqref{eq:26}}{\geq}\frac{\sigma_{q}}{q} \left(  \frac{p!}{L_p^h}\| \nabla h(x_{k+1};x_k) \|_* \right)^{\frac{q}{p}} \\
		& \overset{\eqref{eq:probConv}}{\geq}   \frac{\sigma_{q}}{q} \left(  \frac{p!}{L_p^h} \right)^{\frac{q}{p}} 
		\left(\frac{q \sigma_{q}^{\frac{1}{q-1}}}{q-1}\left(f\left(x_{k+1}\right)-f(x^*)\right)\right)^{\frac{q-1}{p}},
	\end{align*}
	which proves both  statements of the theorem.  
\end{proof}

\noindent Note that if $p>q-1$, then the sequence generated by GHOM has  local \textit{superlinear}  convergence in function values. Indeed,  from Theorem \ref{th:conv-conv} we have $f\left(x_{k}\right)-f(x^*) \to 0$ as $k \to \infty$ and consequently the right hand side  term   $\beta_k =  (q - 1) q^{\frac{p-q+1}{q-1}}  \left(\frac{1}{\sigma_{q}}\right)^{\frac{p+1}{q-1}}   \times \\ \left(\frac{L_{p}^h}{p !}\right)^{\frac{q}{q-1}}  \left(f\left(x_{k}\right) - f(x^*) \right)^{\frac{p}{q-1}-1}$	also converges to zero, provided that  $\frac{p}{q-1}>1$.  E.g., if $q = 2$ ($f$ strongly convex) and $p = 2$, then  the rate is quadratic. If $q = 2$ and $p = 3$, then the  rate  is cubic.  Next, we show local superlinear convergence for a sequence of subgradients of the  convex objective function in the problem	\eqref{eq:optpb}.

\begin{theorem}
	\label{th:conv-conv_super_grad} 
	Under the assumptions of Theorem \ref{th:conv-conv_super0}  the sequence of the norm of  subgradients  $-  \nabla h(x_{k};x_{k-1}) \in \partial (f+\textbf{1}_\mathcal{X})  (x_{k}) $ has the following  convergence rate:
\begin{align*}
	\| \nabla h(x_{k+1};x_k) \|_* \leq  \frac{L_p^h}{p!}   \left(  \frac{q}{\sigma_{q}}   \right)^{\frac{p}{q-1}}     \| \nabla h(x_{k};x_{k-1})\|_*^{\frac{p}{q-1}} \quad \forall k \geq 1.
\end{align*}
Additionally, the residual  of  the first-order optimality conditions for \eqref{eq:optpb} satisfies:
\begin{align*}
 \inf_{x \in \mathcal{X}}  \frac{Df(x_{k+1})[x - x_{k+1}]}{\|x - x_{k+1}\|} \geq  - \|   \nabla h(x_{k+1};x_{k}) \| \quad \forall k\geq 0.
\end{align*} 
\end{theorem}

\begin{proof} 
From Lemma  \ref{lema:2} we have	$-  \nabla h(x_{k};x_{k-1}) \in \partial (f+\textbf{1}_\mathcal{X})  (x_{k}) $ for all $k \geq 1$.  Since GHOM is a descent method and $f$ is uniformly convex, we further have:
	\begin{equation*}
		f(x_k) \overset{\eqref{eq:main}}{\geq} f(x_{k+1}) \overset{\eqref{eq:unifConv}}{\geq} f(x_k) + \left\langle -  \nabla h(x_{k};x_{k-1}), x_{k+1} - x_k  \right\rangle + \frac{\sigma_{q}}{q} \|  x_{k+1} - x_k \|^{q} \;\; \forall k \geq 1. 
	\end{equation*}	
Using the Cauchy-Schwarz inequality we further get: 	
	\begin{align*}
		0  & \geq   \left\langle -  \nabla h(x_{k};x_{k-1}), x_{k+1}  -  x_k  \right\rangle  +  \frac{\sigma_{q}}{q} \|  x_{k+1}  -  x_k \|^{q} \\
		&  \geq  - \|  \nabla h(x_{k};x_{k-1})\|_*  \| x_{k+1}  -  x_k \|  +  \frac{\sigma_{q}}{q} \|  x_{k+1}  -  x_k \|^{q}, 
	\end{align*}	
	or equivalently
	\begin{equation}
		\label{eq:gradrel1}
		\| \nabla h(x_{k};x_{k-1})\|_* \geq  \frac{\sigma_{q}}{q} \|  x_{k+1} - x_k \|^{q-1}  \quad \forall k \geq 0. 	
	\end{equation}	
 Using again {Lemma \ref{lema:2}}, we get:
	\begin{align*}
	 \| \nabla h(x_{k+1};x_k) \|_* \overset{\eqref{eq:26}}{\leq}  \frac{L_p^h}{p!} \| x_{k+1} -x_k \|^p   \overset{\eqref{eq:gradrel1}}{ \leq}   \frac{L_p^h}{p!}   \left(  \frac{q}{\sigma_{q}}   \right)^{\frac{p}{q-1}}     \| \nabla h(x_{k};x_{k-1})\|_*^{\frac{p}{q-1}}.
	\end{align*}
Since according to Theorem 	\ref{th:conv-conv_super0} we have   $\|  x_{k+1} - x_k\| $ converges to zero, the previous relation shows that the sequence of subgradients $-  \nabla h(x_{k+1};x_{k}) \in \partial (f+\textbf{1}_\mathcal{X})  (x_{k+1}) $ converges also to zero. This proves the first statement of the theorem. 

\medskip 

\noindent Further, since we assumed that $\mathcal{X} \subseteq \text{dom}\, f$  has nonempty  relative interior, then the following holds \cite{Roc:70}: $ \partial (f+\textbf{1}_\mathcal{X})  (x_{k+1}) =  \partial f  (x_{k+1})  + \partial \textbf{1}_\mathcal{X}  (x_{k+1})  $. Since  we have already proved that $-  \nabla h(x_{k+1};x_{k}) \in \partial (f+\textbf{1}_\mathcal{X})  (x_{k+1})$, it follows that $-  \nabla h(x_{k+1};x_{k}) \in   \partial f  (x_{k+1})  + \partial \textbf{1}_\mathcal{X}  (x_{k+1})  $ and thus  there exists $f^{x_{k+1}} \in \partial f(x_{k+1})$ such that:
\begin{align*}  
	\langle  f^{x_{k+1}}, x - x_{k+1}  \rangle & \geq   \langle   \nabla h(x_{k+1};x_{k}),  x_{k+1} - x  \rangle \\
	&  \geq - \|   \nabla h(x_{k+1};x_{k})  \|  \| x - x_{k+1} \|  \quad \forall x \in \mathcal{X}, k \geq 0,
\end{align*}
which yields the second statement of the theorem as $f$ is convex and $Df(x_{k+1})[x - x_{k+1}] = \max_{ \bar{f}^{x_{k+1}} \in \partial f(x_{k+1})}  \langle  \bar{f}^{x_{k+1}}, x - x_{k+1}  \rangle$, see \cite{Roc:70}.  
\end{proof}

\noindent From  Theorem \ref{th:conv-conv_super_grad} we conclude that for $\frac{p}{q-1}>1$ the sequence of  subgradients $-  \nabla h(x_{k+1};x_{k}) \in \partial (f+\textbf{1}_\mathcal{X})  (x_{k+1}) $  converges to zero at a local superlinear  rate and it is of the same  order as that for the residual of the function values.  

\medskip 

\noindent Our  superlinear  convergence results from this Section \ref{subsection3.2} have been derived under the  assumption that $h$ is convex.  Note that from Lemma 	\ref{lema:conv} we can see that   by choosing appropriately $M_p$  in Examples \ref{expl:8}, \ref{expl:88},  \ref{expl:4}, \ref{expl:5} and   $M_{p+1}$ in Example  \ref{expl:6}, we obtain that the error functions, $h$, in these examples  are indeed convex. However, if we remove the convexity assumption on $h$ we can still prove superlinear convergence for GHOM in function values, but the rate is slightly worse.  This result is stated~next. 


\begin{theorem}
\label{th:conv-conv_super}
Let the  objective function $f$ in the optimization problem \eqref{eq:optpb} be proper, lower semicontinuous and  uniformly convex  of degree $q \geq 2$ with constant $\sigma_{q}$  on  the closed convex set $\mathcal{X}$.  Additionally,  \red{let Assumptions \ref{def:sur}  and  \ref{ass:opt}	hold}.  Then, the sequence $(x_k)_{k \geq 0}$ generated by algorithm GHOM has the following  convergence rate in function values:
	\begin{align}
		\label{theq:conv-conv_super}
		f(x_{k+1})  - f(x^*) & \leq     \frac{L_p^h}{(p+1)!}  \left(  \frac{q}{\sigma_q} \right)^{\frac{p+1}{q}}  \!\! (f(x_k) - f(x^*))^{\frac{p+1}{q} }.
	\end{align}
\end{theorem}

\begin{proof}
	If $f$ is uniformly convex on the convex set $\mathcal{X}$, then it has a unique optimal point $x^* \in \mathcal{X}$. Moreover, from \eqref{eq:main}, we have:
	\begin{align}
		f(x_{k+1}) & \leq  \min_{y \in \mathcal{X}} (f(y) + \frac{L_p^h}{(p+1)!} \| y-x_k\|^{p+1})  \nonumber \\ 
		& \overset{y=x^*}{\leq}   f(x^*)    + \frac{L_p^h}{(p+1)!} \| x_k - x^*\|^{p+1}. 
		\label{eq:relatii1}
	\end{align}
Further, since $f$ is uniformly convex of degree $q$, then using 
	Lemma 	\ref{lema:uc} for $f$ and the fact that {$ \langle f^{x^*}, x_k - x^* \rangle \geq 0$},  that $x^*, x_k \in \mathcal{X} $ and that  $f^{x^*} \in \partial f(x^*)$, we~get: 
	\begin{align}
		\label{eq:relatii2}
		f(x_k) - f(x^*)   \geq \frac{\sigma_{q}}{q}\|x_k - x^*\|^{q}.  
	\end{align}
	Combining the inequalities \eqref{eq:relatii1} and \eqref{eq:relatii2}, we further obtain:
	\begin{align*}
		f(x_{k+1})  - f(x^*) & \leq \frac{L_p^h}{(p+1)!} \| x_k - x^*\|^{p+1} \leq   \frac{L_p^h}{(p+1)!} \left(  \frac{q}{\sigma_q} (f(x_k) - f(x^*) ) \right)^{\frac{p+1}{q}} \\
		&  =  \left(  \frac{L_p^h}{(p+1)!}  \left(  \frac{q}{\sigma_q} \right)^{\frac{p+1}{q}}  \!\! (f(x_k) - f(x^*) )^{\frac{p+1}{q} -1}  \right ) (f(x_k) - f(x^*)),
	\end{align*}
thus proving the statement of the theorem.
\end{proof}

\begin{remark}
	\noindent If $q<p+1$ in  \eqref{theq:conv-conv_super}, we have that $\beta_k =  \frac{L_p^h}{(p+1)!}  \left(  \frac{q}{\sigma_q} \right)^{\frac{p+1}{q}}   (f(x_k) - f(x^*) )^{\frac{p+1}{q} -1} $	converges to zero, since from Theorem \ref{th:conv-conv} we have $f\left(x_{k}\right)-f(x^*) \to 0$,  thus obtaining again  superlinear convergence in function values.  Finally,  using  the inequality  \eqref{eq:relatii2}  in the convergence rates \eqref{theq:conv-conv_super0} and \eqref{theq:conv-conv_super}, respectively,   we immediately  obtain   superlinear convergence also for   $\|x_k - x^*\|$.  Since the derivations are straightforward, we omit them.  
\end{remark}


\section{ Convergence analysis  of GHOM for  nonconvex optimization}
\label{sec:non-conv} 
In this  section we analyze the convergence behavior of algorithm GHOM for solving the nonconvex nonsmooth optimization problem  \eqref{eq:optpb}.  Hence, in this section we assume a general (possibly nonconvex) proper lsc  objective function $f: \Eb \to \bar{ \mathbb{R} }$ and a  closed convex set $\mathcal{X} \subseteq \text{dom} f$.  In these general settings,   we only assume that $x_{k+1}$ is  a stationary point of  subproblem \eqref{eq:sp} satisfying, additionally,   descent \eqref{eq:descGhom} (i.e.,  $x_{k+1}$ satisfies  $D g(x_{k+1};{x}_{k})[y -x_{k+1}]  \geq 0 \;\;  \forall y \in \mathcal{X}$ and  \eqref{eq:descGhom}). Let us  define the notion of asymptotic stationary points generated by GHOM algorithm for problem \eqref{eq:optpb}. 

\begin{definition} 
	Assume that  the directional derivatives of the function $f$ at any  $x \in \mathcal{X}$ along $y-x$ exist  for all $y \in \mathcal{X}$, denoted  $D f(x) [y-x] $. Then, a sequence $\left(x_{k}\right)_{k\geq 0}$ satisfies the \textit{asymptotic first-order optimality conditions}  for the nonconvex problem \eqref{eq:optpb} if the following holds:
	\[
	\liminf _{k \rightarrow+\infty} \inf _{x  \in \mathcal{X}} \frac{D f(x_{k}) [x-x_{k}]}{\left\|x-x_{k}\right\|} \geq 0.		\]
	\qed 
\end{definition}

\noindent Note that if $f$ is differentiable  on $\mathcal{X}= \mathbb{E}$, then the directional derivative $D f(x)[y - x] = \langle \nabla f(x), y - x \rangle$ and the asymptotic  first-order optimality conditions imply that the sequence $(\nabla f(x_k))_{k \geq 0}$ converges to $0$. 

\subsection{Global convergence of GHOM for  nonconvex optimization}
In this section we analyze  the global convergence behavior of  the iterates of GHOM in the general nonconvex settings. We first derive some auxiliary~result:
\begin{lemma}
	\label{lemma:nabla2}
	Let $\tilde h$ be a function $p \geq 2$ differentiable and with the $p$ derivative smooth with constant $L_p^{\tilde h}$. For any $x, y \in \text{dom} \, \tilde h$  and  scalar $M_p^{\tilde h} \geq L_p^{\tilde h}$ let us define:
	\[  H =  \nabla^2 \tilde h (x) +  \sum_{i=3}^p \frac{1}{(i-1)!} D^i \tilde h (x) [y-x]^{i-2} +  \frac{M_p^{\tilde h}}{p!}  \|y-x\|^{p-1} B.   \]	
	Then, we have the following bounds on the matrix $H$:
	\begin{align*}  
		\int_{0}^1  \nabla^2 & \tilde h (x + \tau(y-x))   d\tau  \preceq  H \\ 
		& \preceq \int_{0}^1 \left( \nabla^2 \tilde h (x + \tau(y-x)) +  \frac{M_p^{\tilde h} + L_p^{\tilde h}}{(p-1)!} \tau^{p-1} \|y-x\|^{p-1} B \right) \! d\tau. \end{align*}
\end{lemma}	

\begin{proof}
	We note that:
	\begin{align*}
		H & =   \int_{0}^1 \left(   \nabla^2 \tilde h (x)  + \sum_{i=3}^p \frac{\tau^{i-2}}{(i-2)!} D^i \tilde h (x) [y-x]^{i-2} +  \frac{M_p^{\tilde h}  \tau^{p-1}}{(p-1)!}  \|y-x\|^{p-1} B \right) \! d\tau  \\
		&  =	 \int_{0}^1 \left(   \nabla^2 T_p^{\tilde h} (x + \tau(y-x);x) +  \frac{M_p^{\tilde h}  \tau^{p-1}}{(p-1)!}  \|y-x\|^{p-1} B \right) \! d\tau \\
		& \overset{\eqref{eq:TayAppG2}}{\preceq}  \int_{0}^1 \left(   \nabla^2 \tilde h (x + \tau(y-x)) +   \frac{M_p^{\tilde h} + L_p^{\tilde h}}{(p-1)!} \tau^{p-1} \|y-x\|^{p-1} B \right) \! d\tau.
	\end{align*}	 	
	Similarly, we have:
	\begin{align*}
		H  & \!\!\overset{\eqref{eq:TayAppG2}}{\succeq}  \!\! \int_{0}^1 \!\!  \left(\!\!   \nabla^2 \tilde h (x \!+\! \tau(y\!-\!x)) \!+\!   \frac{M_p^{\tilde h} \!-\! L_p^{\tilde h}}{(p-1)!} \tau^{p-1} \|y-x\|^{p-1} B \right) \! d\tau \! \succeq\!  \int_{0}^1   \!\!\nabla^2 \tilde h (x \!+\! \tau(y\!-\!x)) d\tau,
	\end{align*}	
	since  $M_p^{\tilde h} \geq L_p^{\tilde h}$ and $B \succeq 0$. These conclude our statement. 
\end{proof}


\noindent Next theorem shows that the sequence generated by GHOM  satisfies the  asymptotic first-order optimality conditions.

\begin{theorem}
\label{th:nonconv-gen_dom}
Let the (possibly   nonconvex) objective function $f$ in the optimization problem \eqref{eq:optpb} be proper, lower semicontinuous,  have directional derivates at any point  $x \in \mathcal{X}$ and satisfy Assumption \ref{def:sur} for a given  $p \geq 1$.  Then, the sequence $\left(x_{k}\right)_{k \geq 0}$  generated by GHOM satisfies the  asymptotic first-order optimality conditions and $\left(f(x_{k})\right)_{k \geq 0}$ monotonically decreases.
\end{theorem}
\begin{proof}
	Using the properties of the surrogate (see  Assumption \ref{def:sur}) and the descent property \eqref{eq:descGhom}, we obtain:
	\begin{equation*}
		f(x_k) = g(x_k;x_k) \geq g(x_{k+1};x_{k}) \geq f(x_{k+1})\quad \forall k \geq 0.
	\end{equation*}
	This relation guarantees that $(f(x_k))_{k>0}$ is nonincreasing sequence and thus convergent, since $f$ is assumed to be bounded from below by $f^*$. Moreover, since we also assume the level set ${\cal L}_f(x_0)$  bounded (see Section \ref{sect:ghom}), then it follows that the sequence generated by GHOM $(x_k)_{k\geq 0} \subset {\cal L}_f(x_0)$   is also bounded. Using further  the definition of the error function $h$:
	\begin{align}
		\label{eq: hf_dom}
		0 \leq h(x_{k+1};x_k) = g(x_{k+1};x_k) -f(x_{k+1}) \leq f(x_k) -f(x_{k+1}).
	\end{align}
	Telescoping the previous relation for $k = 0:\infty$, we get:
	\begin{equation*}
		0 \leq \sum_{k = 0}^{\infty} h(x_{k+1};x_k) \leq f(x_0) -f^* < \infty.
	\end{equation*}
	Therefore, the positive term of the series, $(h(x_{k+1}; x_k))_{k \geq 0}$, necessarily converges to 0. Since $g = h+ f$, it also follows that  the sequence $(g(x_{k+1}; x_k))_{k \geq 0}$ converges to the same limit as the sequence $(f(x_k))_{k>0}$. 

\medskip 
		
\noindent For proving the asymptotic  first-order optimality conditions we consider two cases: $p \geq 2$ and $p=1$.   First, for  $p \geq 2$,  from the optimality of $x_{k+1}$ we have that:
$$
D g(x_{k+1};{x}_{k})[y -x_{k+1}]  \geq 0 \quad \forall y \in \mathcal{X}.
$$
Further, from the definition of the error function $h= g-f$, its differentiability  and the fact that $f$ is assumed to have directional derivatives at any point $x \in \mathcal{X}$, we have from  calculus rules that \cite{Roc:70}[Theorem 10.1]:
\begin{align*}
	D f(x_{k+1})&[y-x_{k+1}]   = D g(x_{k+1}; {x}_{k})[y-x_{k+1}] - D h(x_{k+1}; {x}_{k})[y-x_{k+1}] \\
	& \geq - D h(x_{k+1}; {x}_{k})[y-x_{k+1}]  = - \langle  \nabla h(x_{k+1}; {x}_{k}),  y- x_{k+1} \rangle \quad  \forall y \in  \mathcal{X}. 
\end{align*}
Using the Cauchy-Schwarz inequality in the previous relation, we get:
\begin{equation}\label{eq:25_domf}
	D f(x_{k+1})[y-x_{k+1}] \geq -\| \nabla h(x_{k+1}; {x}_k)\|_* \, \|y- x_{k+1} \|\quad \forall y \in  \mathcal{X}.
\end{equation}
According to Assumption \ref{def:sur},  the error function $ h(\cdot;x_k)$ has the $p$ derivative smooth with Lipschitz constant $ L_p^h$ on $\mathcal{X}$.		Further, let us consider the  following auxiliary  point $y_{k+1}$  (note that in practice we do not need to compute $y_{k+1}$):  
$$y_{k+1} = \argmin_{y \in \text{dom} \,  h(\cdot;x_k)} T_p^{{h}}(y;x_{k+1}) + \frac{M_p^h}{(p+1)!} \|y-x_{k+1} \|^{p+1},$$
where  $M_p^h > L_p^h$ (recall our notation: $T_p^{{h}}(y;x_{k+1})$ means  the Taylor approximation of the function $ h(\cdot;x_k)$  around $x_{k+1}$ and evaluated at $y$, i.e.,  $T_p^{{h}}(y;x_{k+1}) = h(x_{k+1};x_k) + \sum_{i=1}^{p} \frac{1}{i !} D^{i} h(x_{k+1};x_k)[y-x_{k+1}]^{i} $).  Using the (global) optimality of $y^{k+1}$ (recall that  in Assumption  \ref{def:sur} we consider $\mathcal{X} \subseteq \text{dom} \,  h(\cdot;x_k)$), we have:
\begin{align*}
	&T_p^{{h}}(y_{k+1};x_{k+1}) + \frac{M_p^h}{(p+1)!} \| y_{k+1}-x_{k+1}\|^{p+1} \\ & \leq T_p^{{h}}(x_{k+1};x_{k+1}) + \frac{M_p^h}{(p+1)!} \| x_{k+1}-x_{k+1}\|^{p+1} \\
	&= h(x_{k+1};x_k) + \sum_{i=1}^{p} \frac{1}{i !} D^{i} h(x_{k+1};x_k)[x_{k+1}-x_{k+1}]^{i} = h(x_{k+1};x_k).
\end{align*} 
Moreover, writting explicitly the left term of the previous  inequality, we get:
\begin{align*}
	& h(x_{k+1};x_k) + \sum_{i=1}^{p} \frac{1}{i!} D^{i} h(x_{k+1};x_k)[y_{k+1}-x_{k+1}]^{i} + \frac{M_p^h}{(p+1)!} \| y_{k+1}-x_{k+1}\|^{p+1} \\ 
	& \leq  h(x_{k+1};x_k).
\end{align*} 
Thus, we have:
\begin{equation}\label{eq:24_domf}
	\sum_{i=1}^{p} \frac{1}{i!} D^{i} h(x_{k+1};x_k)[y_{k+1}-x_{k+1}]^{i} + \frac{M_p^h}{(p+1)!} \| y_{k+1}-x_{k+1}\|^{p+1} \leq 0.
\end{equation}
From relation \eqref{eq:TayAppBound} we also obtain:
\begin{equation*}
  h(y;x_k) \leq T_p^{{h}}(y;x_{k+1}) + \frac{L_p^h}{(p+1)!} \| y-x_{k+1}\|^{p+1} \quad \forall y  \in \text{dom} \,  h(\cdot;x_k).
\end{equation*}
We rewrite this relation for our chosen point $y_{k+1}$:
\begin{align*}
	& h(y_{k+1};x_k) \leq T_p^{{h}}(y_{k+1};x_{k+1}) + \frac{L_p^h}{(p+1)!} \| y_{k+1}-x_{k+1}\|^{p+1} \\
	& = h(x_{k+1};x_k) \!+\!  \sum_{i=1}^{p} \! \frac{1}{i!} D^{i} h(x_{k+1};x_k)[y_{k+1} \!-\! x_{k+1}]^{i}  \!+\!  \frac{L_p^h}{(p+1)!} \| y_{k+1}-x_{k+1}\|^{p+1} \\
	&\overset{\eqref{eq:24_domf}}{\leq}h(x_{k+1};x_k) - \frac{M_p^h -L_p^h}{(p+1)!} \| y_{k+1}-x_{k+1}\|^{p+1}.
\end{align*}
Recalling that $h$ is nonnegative, then it follows that:
\begin{align}
	\label{eq: hy_dom}
	0 \leq h(x_{k+1};x_k) - \frac{M_p^h -L_p^h}{(p+1)!} \| y_{k+1}-x_{k+1}\|^{p+1}.
\end{align}
This leads to:
\begin{equation*}
	h(x_{k+1};x_k) \geq  \frac{M_p^h -L_p^h}{(p+1)!} \| y_{k+1}-x_{k+1}\|^{p+1} \geq 0.
\end{equation*}		
Since $(h(x_{k+1}; x_k))_{k \geq 0}$ converges to 0, then necessarily $(y_{k+1} - x_{k+1})_{k \geq 0}$ converges to $0$ and is bounded, since $h$ is continuous and $(x_{k})_{k \geq 0}$ is bounded (recall that we assume that the sublevel set  $\mathcal{L}_f(x_0)$ is bounded). Consequently, the sequence  $(y_{k})_{k \geq 0}$ is also bounded. Moreover,  from the first-order necessary  optimality conditions for  $y_{k+1}$ (recall that according to Assumption \ref{def:sur}, $\text{dom} \, h (\cdot;x_k)$ is an open set), we have:
\begin{align}
	\label{eq:optcondy_domf}
	\nabla h(x_{k+1};x_{k}) + H_{k+1}[y_{k+1}-x_{k+1}] = 0,
\end{align}
where  we denote the matrix  
\begin{align}
	\label{eq:hk+1_dom}
	H_{k+1} & = \nabla^2 h(x_{k+1};x_k) + \sum_{i=3}^{p} \frac{1}{(i-1)!} D^i  h(x_{k+1};x_k)[y_{k+1}-x_{k+1}]^{i-2} \\ 
	& \qquad + \frac{M_p^h}{p!} \| y_{k+1} -x_{k+1}\|^{p-1} B.  \nonumber 
\end{align}
From Lemma \ref{lemma:nabla2},	 we have that:
	\begin{align}  
		\label{eq:Hyx_dom}	
		&	\int_{0}^1  \nabla^2  h (x_{k+1} + \tau(y_{k+1}-x_{k+1});x_k)   d\tau  \preceq  H_{k+1} \\ 
		& \preceq \int_{0}^1 \! \left( \nabla^2  h (x_{k+1} \!+\! \tau(y_{k+1} \!-\! x_{k+1});x_k) \!+\!  \frac{M_p^{h} \!+\! L_p^{h}}{(p-1)!} \tau^{p-1} \|y_{k+1} \!-\! x_{k+1}\|^{p-1} \!B \! \right) \! d\tau. \nonumber 
	\end{align}
	Since the  sequences $(x_{k})_{k>0}$  and $(y_{k})_{k>0}$ are  bounded and $h$ is $p \geq 2$ times continuously  differentiable, then $ \nabla^{2} h(x_{k+1} + \tau(y_{k+1}-x_{k+1}); x_k)$ is bounded for $\tau \in [0, 1]$.  Moreover,  $y_{k+1} -x_{k+1}  \rightarrow 0$ as $k \to \infty$ and it is bounded. Therefore, $H_{k+1}$ is bounded and  consequently from \eqref{eq:optcondy_domf} it follows that 
	\begin{align}
		\label{eq:gradh0_domf}	
		\nabla h(x_{k+1};x_{k}) \to 0  \; \text{as}  \; k \to \infty. 
	\end{align}  
	\noindent For the case $p=1$ we can just take $y_{k+1} = x_{k+1} - 1/L_1^h \nabla h (x_{k+1};x_k)$. Then, using that $h(\cdot;x_k)$ has gradient Lipschitz  with constant $L_1^h$ we obtain \cite{Nes:04}:
	\[   0 \leq h(y_{k+1};x_k)  \leq h(x_{k+1}; x_k) - \frac{1}{2 L_1^h} \| \nabla h(x_{k+1};x_k) \|^2_*, \]
	which further yields
	\[  \frac{1}{2 L_1^h} \| \nabla h(x_{k+1};x_k) \|^2_*  \leq h(x_{k+1};x_k) \to 0  \; \text{as}  \; k \to \infty, \]
	since we  have already  proved that  the sequence $h(x_{k+1};x_k)$ converges to zero. Therefore, also in the case $p=1$ we have  \eqref{eq:gradh0_domf} valid.  Finally, using  \eqref{eq:gradh0_domf} in \eqref{eq:25_domf},  minimizing over $y \in \mathcal{X}$ and taking the infimum limit,  we get that the sequence $\left(x_{k}\right)_{k \geq 0}$ satisfies the asymptotic  first-order optimality conditions for the nonconvex problem \eqref{eq:optpb}. This concludes our statements.  
\end{proof}

\noindent {Note that the main difficulty in the previous proof is to handle  $h$ having   $p \geq 2$ derivative smooth. We overcome  this difficulty by introducing a new sequence $(y_k)_{k \geq 0}$, proving that it has similar properties as the sequence  $(x_k)_{k \geq 0}$ generated by GHOM, and then  using the  optimality conditions for $y_k$ instead of $x_k$.  Note that in practice we do not need to compute the auxiliary sequence  $y_k$.}  Moreover, if $ \mathcal{X}  = \Eb$, the previous theorem states that the minimum norm sequence of  points from the limiting subdifferential  $(\partial f(x_{k}))_{k \geq 0}$,  denoted $S(x_k)$,   converges to $0$ as $k \to \infty$.


\medskip 

\noindent   Obviously,  the previous theorem  does not yield  convergence rate for    the residual  of  the first-order optimality conditions.    Hence,  the requirements from  Assumption  \ref{def:sur} on the surrogate function $g$  do not seem to be reach enough to enable  convergence rate in the nonconex settings. More exactly, the bound $h(y;x) \geq 0$ for any $y \in \mathcal{X}$ is not sufficiently tight when analysing  convergence of  the residual of  the first-order optimality conditions.   However,   at a  closer look one can notice that the surrogate $g$ of Example   \ref{expl:4} induces the following inequality on the error function $h$:
\begin{align*}
	h(y;x)  \!=\!   T_p^f(y;x) + \frac{M_p }{(p+1)!}   \|y - x\|^{p+1} - f(y)   \!\geq\!  \frac{M_p - L_p^f}{(p+1)!}   \|y - x\|^{p+1} \;\;  \forall x, y \in \mathcal{X},	
\end{align*}	
where  $M_p > L_p^f$.  In fact, using the same reasoning,  it is easy to see that such a relation holds for all the surrogate functions from Examples \ref{expl:8},  \ref{expl:88}, \ref{expl:4}, \ref{expl:5},  \ref{expl:888} and 	\ref{expl:6}.  Hence, let us additionally assume that  our surrogate function satisfies  the  following inequality in terms of the error function $h$ for some positive constant $C_p >0$: 
\begin{align}
	\label{eq:unifh_domf}
	C_p   \|y - x\|^{p+1} \leq 	h(y;x) \; \left( := g(y;x) - f(y) \right)   \quad \forall x,y \in  \mathcal{X}. 
\end{align}	
Hence,  we prove below that if we require additionally the condition \eqref{eq:unifh_domf}, we can get convergence rate  to a
point from which there is no descent direction, i.e.,   we can strengthen the results from Theorem~\ref{th:nonconv-gen_dom}.  

\begin{theorem}
\label{th:noncf1f2}
Let  the assumptions from Theorem \ref{th:nonconv-gen_dom} hold. Additionally, assume that our surrogate function satisfies the relation \eqref{eq:unifh_domf}. Then,  there exists an iteration index $i_* \in \{0,\cdots, k-1\}$ such that the following convergence rate in terms of first-order optimality conditions holds:
	\begin{equation*}
	\inf _{x  \in \mathcal{X}} \frac{D f(x_{i^*}) [x-x_{i^*}]}{\left\|x-x_{i^*}\right\|}     \geq  -  \frac{L_p^{h}}{p!}\left(\frac{(f(x_0) - f^*)}{C_p k }\right)^{\frac{p}{p+1}}. 
	\end{equation*}
\end{theorem}

\begin{proof}
Since the error function $h$ has the $p$ derivative Lipschitz, we have:
\[  \|\nabla h(x_{k+1};x_k)   -   \nabla T_p^h(x_{k+1};x_k) \|_* \leq \frac{L_p^h}{p!} \| x_{k+1} - x_k \|^p.  \]
Using that $\nabla T_p^h(x_{k+1};x_k) = 0$ (according to Assumption \ref{def:sur} (iii)), we get:
 \[  \|\nabla h(x_{k+1};x_k) \|_* \leq \frac{L_p^h}{p!} \| x_{k+1} - x_k \|^p.  \]
Combining this relation with  our assumptions, we further obtain:
\begin{align*}
& f(x_k)	- f(x_{k+1}) \overset{\eqref{eq: hf_dom}}{\geq} h(x_{k+1};x_k)    \overset{\eqref{eq:unifh_domf}}{\geq}  C_p   \|x_{k+1} - x_k\|^{p+1} \\	
& =   C_p   \left( \|x_{k+1} - x_k\|^{p} \right)^{\frac{p+1}{p}}  \geq   C_p   \left( \frac{p!}{L_p^h}   \|\nabla h(x_{k+1};x_k) \|_*  \right)^{\frac{p+1}{p}}. 
\end{align*}	
Telescoping from $i=0:k-1$ the previous inequality, we further get:
\begin{align*}
	f(x_{0}) -f^* &\!\geq C_p  \sum_{i=0}^{k-1} \left( \frac{p!}{L_p^h}   \|\nabla h(x_{i+1};x_i) \| _* \right)^{\frac{p+1}{p}}  \!\!\!=\!   C_p \left( \frac{p!}{L_p^h}   \right)^{\frac{p+1}{p}}  \sum_{i=0}^{k-1}  	 \|\nabla h(x_{i+1};x_i) \| _*^{\frac{p+1}{p}}\\
	&\geq k \cdot C_p \left( \frac{p!}{L_p^h}   \right)^{\frac{p+1}{p}}   \min_{i=0:k-1} \|\nabla h(x_{i+1};x_i) \| _*^{\frac{p+1}{p}}
\end{align*}
or equivalently 
\begin{align*}
 \min_{i=0:k-1} \|\nabla h(x_{i+1};x_i) \| _*  \leq  \frac{L_p^h}{p!} \left(  \frac{f(x_0) - f^*}{C_p k} \right)^{\frac{p}{p+1}}. 
\end{align*}
Let $i_* \in \{0,\cdots, k-1\}$ be the iteration index where the minimum in the previous relation is achieved. Then, combining the  previous relation with \eqref{eq:25_domf} and minimizing over $y \in  \mathcal{X}$,  we obtain our statement.   
\end{proof}	

\noindent It is known that for general nonconvex problems deciding whether a descent direction exists from a point is NP-hard \cite{Nes:13}.  However, for an adaptive regularization  algorithm (called AR$p$) on problems with smooth objective and simple constraints  \cite{cartis2017,cartis:2020}  or for a first-order algorithm on problems with composite structure \cite{Nes:13} similar convergence  rates as in Theorem 	\ref{th:noncf1f2} have been derived.  The results for GHOM algorithm   on global convergence  to a point from which there is no descent direction (including convergence rates) from Theorems 	\ref{th:nonconv-gen_dom} and  	\ref{th:noncf1f2} are more general  than the ones from literature as they cover more complicated objective functions (e.g., general composite models or nonsmooth objective functions) and many new surrogate functions (see Examples \ref{expl:8},  \ref{expl:88},   \ref{expl:5}, \ref{expl:888}  and \ref{expl:6}).  The only condition that we need is that the ojective function admits a higher-order surrogate (see Assumption  \ref{def:sur}).


\subsection{Local convergence of GHOM under the  KL property}
The KL property was first analyzed in details in \cite{BolDan:07} and then widely applied to analyze the convergence behavior of various first-order  \cite{AttBol:09,BolDan:07} and  second-order  \cite{FraGar:15,ZhoWan:18} algorithms  for nonconvex optimization. \textit{However, to the best of our knowledge there are no studies analyzing  the convergence rate of higher-order majorization-minimization algorithms under the KL property}. The main difficulty comes from the fact that $f$ satisfies the KL property while the error function $h$ is assumed smooth and it is hard to establish connections between them. In the next theorem we connect the geometric property of nonconvex function $f$ with the smooth property of $h$  through  the point $y_{k+1}$ and establish local convergence of GHOM in the full parameter regime  of the KL property.    Let us  denote the set of limit points of the sequence $(x_k)_{k \geq 0}$ generated by algorithm  GHOM with  $\Omega(x_0)$. 
\begin{lemma}
	\label{lemma:kl}
	Let  the assumptions from Theorem \ref{th:nonconv-gen_dom} hold and, additionally,  assume  that  $f$ is continuous.  Then,  $\Omega(x_0)$ is compact set and  $f$ is constant on $\Omega(x_0)$.
\end{lemma}

\begin{proof}
	From Theorem \ref{th:nonconv-gen_dom} we have  that the sequence $(x_k)_{k \geq 0}$ is bounded, hence the set of limit points $\Omega(x_0)$ is also bounded.  Closedness of  $\Omega(x_0)$ also follows by observing that $\Omega(x_0)$ can be viewed as an intersection of closed sets, i.e.  $\Omega(x_0) = \cap_{j \geq 0} \overline{ \cup_{k \geq j} \{x_k\} }$. Hence, $\Omega(x_0)$ is a compact set  and $\text{dist}(x_k, \Omega(x_0)) \to 0$ as $k \to \infty$.   Further, let us also show that  $f(\Omega(x_0))$ is constant. From Theorem \ref{th:nonconv-gen_dom} we have that $\left(f(x_{k})\right)_{k \geq 0}$ is monotonically decreases and since $f$ is assumed bounded from below by $f^* > -\infty$, it converges, let us say to $ f_* > -\infty$, i.e. , $f(x_{k}) \to  f_*$ as $k \to \infty$. On the other hand let $x_*$ be a limit point of the sequence $(x_k)_{k \geq 0}$, i.e.,   $x_* \in \Omega(x_0)$.  This means that there is a subsequence   $(x_{k_j})_{j \geq 0}$ such that $ x_{k_j} \to x_*$ as $j \to \infty$.   If $f$ is continuous, then   $\lim_{j \to \infty}  f(x_{k_j}) = f(x_*)$ and  $ f(x_*)  = \lim_{j \to \infty}  f(x_{k_j}) = f_*$. 
	 Hence,  we get $ f(\Omega(x_0)) = f_*$.   
\end{proof}

\medskip

\noindent Note that   in order to derive rates we need to consider explicit form for the function $\kappa$ in Definition \ref{def:kl}. Below we consider the case of semi-algebraic functions, i.e., we assume that $f$ satisfies the KL property \eqref{eq:kl}  (from previous lemma we note that  all the conditions  of the KL property \eqref{eq:kl} are satisfied). 

\begin{theorem}
\label{th:nonconv-gen-kl}
Let the (possibly   nonconvex) objective function $f$ in the optimization problem \eqref{eq:optpb} be proper, continuous,  have directional derivates at any point  $x \in \mathcal{X}$ and satisfy Assumption \ref{def:sur} for a given  $p \geq 1$. Additionally, assume that $f$  satisfies the KL property \eqref{eq:kl} for some $r>1$.   Then,  the sequence  $(x_k)_{k \geq 0}$ generated by GHOM  satisfies:
	\begin{enumerate}
		\item If $r > p+1$, then $f(x_k)$ converges locally to $f_*$ at a superlinear rate.
		\item If $r = p+1$, then $f(x_k)$ converges locally to $f_*$ at a linear rate.
		\item  If $r < p+1$, then $f(x_k)$ converges locally to $f_*$ at a sublinear rate.	
	\end{enumerate}	  
\end{theorem}

\begin{proof}
	Note that the set of limit points $\Omega(x_0)$ of the sequence $(x_k)_{k \geq 0}$ generated by GHOM is compact, $\text{dist}(x_k, \Omega(x_0)) \to 0$ as $k \to \infty$ and $f(\Omega(x_0))$ is constant taking value $f_*$ (see Lemma 	\ref{lemma:kl}).  Moreover, we have that $\left(f(x_{k})\right)_{k \geq 0}$  monotonically decreases and converges to $ f_*$.	Then, for any $\delta, \epsilon >0$ there exists $k_0$ such that:
	\[   x_k \in \{x \in \mathcal{X}: \text{dist}(x, \Omega(x_0)) \leq \delta, f_* < f(x) < f_* + \epsilon  \}  \quad \forall k \geq k_0 \; \text{ and } \;  \eqref{eq:kl} \; \text{ holds}. \]
	Hence, all the conditions of the KL property from  Definition \ref{def:kl}  are satisfied and we can exploit the KL inequality  \eqref{eq:kl}.  Combining 	\eqref{eq: hf_dom} with 	\eqref{eq: hy_dom} we get:
		\vspace*{-0.2cm}
	\begin{align}
		\label{eq:hfy}
		\frac{M_p^h - L_p^h}{(p+1)!} \| y_{k+1} - x_{k+1}\|^{p+1} \leq f(x_k) - f(x_{k+1}) \quad \forall k \geq 0, 	
	\end{align}	
for any fixed  $M_p^h > L_p^h$. From  $x_{k+1}$ being a stationary point of subproblem 	\eqref{eq:sp},  we have that  $D g(x_{k+1};{x}_{k})[y -x_{k+1}]  \geq 0$ for all $y \in \mathcal{X}$.  From the relation  $h= g-f$, using that $h$ is differentiable  we further get  that $D h(x_{k+1};{x}_{k}) [y -x_{k+1}]  + D f(x_{k+1}) [y -x_{k+1}] \geq 0$   for all $y \in \mathcal{X}$. Moreover, using that $h$ is differentiable,   $D f(x_{k+1}) [y -x_{k+1}] = \sup_{f^{x_{k+1}} \in \partial f(x_{k+1})} \langle  f^{x_{k+1}} ,  y -x_{k+1}  \rangle$ and $\mathcal{X}$ is convex set, from calculs rules  we get that  $- \nabla h(x_{k+1};x_k)  \in   \partial (f + \textbf{1}_{\mathcal{X}})(x_{k+1}) $.   Then, using 	\eqref{eq:optcondy_domf} and the definition of  $H_{k+1}$ from 	\eqref{eq:hk+1_dom}, 
	we obtain: 
	\begin{align} 
		\label{eq:nfyx} 
		S(x_{k+1})  & = \text{dist} (0, \partial (f + \textbf{1}_{\mathcal{X}})(x_{k+1}))   \leq   \| \nabla h(x_{k+1};x_{k}) \|_*  \nonumber \\
		& =  \| H_{k+1}(y_{k+1}-x_{k+1}) \|_*  \leq c_{\max} \cdot \| y_{k+1}-x_{k+1} \|  \quad \forall k \geq 0, 
	\end{align}  
	where $c_{\max} = \max_{k \geq 0} \| H_{k+1}\| < \infty$ (see  \eqref{eq:Hyx_dom}).  From the KL property \eqref{eq:kl}, we have:
	\begin{align*}
		f(x_{k+1})	- f_* & \leq \sigma_r  S(x_{k+1})^r \overset{\eqref{eq:nfyx} }{\leq}  \sigma_r  c_{\max}^r  \| y_{k+1}-x_{k+1}  \|^r  \\
		& \overset{\eqref{eq:hfy} }{\leq}  \sigma_r  c^r_{\max}  \left( \frac{(p+1)!}{M_p^h - L_p^h}  \right)^{\frac{r}{p+1}}  (f(x_k) - f(x_{k+1}))^{\frac{r}{p+1}}  \quad \forall k \geq k_0.
	\end{align*}		
\noindent 	Let us denote $\Delta_k = f(x_k) - f_*$ and $C_{\max} = \sigma_r  c^r_{\max}  \left( \frac{(p+1)!}{M_p^h - L_p^h}  \right)^{\frac{r}{p+1}} $. Then, we obtain:
	\begin{align}
		\label{eq:lkyx0}	
		\Delta_{k+1} \leq C_{\max} (\Delta_k - \Delta_{k+1})^{\frac{r}{p+1}}   \quad \forall k \geq k_0. 
	\end{align}

	\vspace{-0.2cm}
	
	\noindent If we denote \red{$  \tilde \Delta_k =  C^{\frac{p+1}{r-p-1}}_{\max}  \Delta_k $}, we further get the recurrence:
	\begin{align}
		\label{eq:lkyx} 
		\tilde \Delta_{k+1}^{\frac{p+1}{r}} \leq \tilde \Delta_k - \tilde  \Delta_{k+1}   \quad \forall k \geq k_0. 
	\end{align}  
	We distinguish the following cases:\\
	\textit{Case (i):} $r > p+1$. Then, from \eqref{eq:lkyx}  we have: 
	\[  \tilde \Delta_{k+1} \leq \left(1 + \tilde \Delta_{k+1}^{\frac{p+1}{r} -1} \right)^{-1}  \tilde \Delta_k,   \]  
		
		\vspace{-0.2cm}
	
	\noindent and since $f(x_k) \to f_*$ it follows that $\tilde \Delta_{k+1} \to 0$ and hence $\tilde \Delta_{k+1}^{\frac{p+1}{r} -1}  \to \infty$ as $\frac{p+1}{r} -1 <0$. Therefore, in this case $f(x_k)$ locally converges  to $f_*$ at a superlinear rate. \\
	\textit{Case (ii):} $r= p+1$.  Then, from \eqref{eq:lkyx0} we have:
	\[  (1 +C_{\max})  \Delta_{k+1} \leq C_{\max} \Delta_{k},  \]
	hence in this case $f(x_k)$ locally converges  to $f_*$ at a linear rate. \\
	\textit{Case (iii):} $r < p+1$.  Then, from \eqref{eq:lkyx} we have:
	\[  \tilde \Delta_{k+1}^{1 + \frac{p+1-r}{r}} \leq \tilde \Delta_k - \tilde  \Delta_{k+1}, \; \text{with} \;\;  \frac{p+1-r}{r} >0, \]
	
		\vspace{-0.2cm}
		
\noindent 	and from Lemma 11 in \cite{Nes:19Inxt} we have for some constant $\alpha>0$: 
	\[    \tilde \Delta_{k}  \leq \frac{   \tilde \Delta_{k_0} }{ (1 + \alpha (k-k_0))^ \frac{r}{p+1-r}} \quad \forall k \geq k_0, \] 
Therefore, in this case $f(x_k)$ locally converges  to $f_*$ at a sublinear rate. 
\end{proof}

\noindent Note that  the superlinear rate  from 	Theorem \ref{th:conv-conv_super0} was derived for uniformly convex functions (i.e., the function $\kappa$ in Definition  \ref{def:kl} is of the form $\kappa(t) = c_\text{unif} \cdot t^{1/q}$, with $q \geq 2$), while the superlinear rate  from Theorem \ref{th:nonconv-gen-kl} was obtained for the class of (possibly nonconvex) semi-algebraic functions (i.e.,  function $\kappa$ in Definition  \ref{def:kl} is of the form $\kappa(t) = c_\text{sa} \cdot t^{1-1/r}$, with $r > 1$).  Consequently, also the proofs in these two theorems are different.


\medskip 

\noindent  Further, for nonconvex unconstrained problems we can further derive convergence rates for the sequence $(x_k)_{k\geq 0}$ generated by GHOM in terms of first- and second-order optimality criteria. We use the notation $a^+ = \max(a, 0)$.

\begin{theorem}
	\label{th:soc}
	Let $f: \rset^n \to \rset$ be a $p>1$ times differentiable function with smooth $p$ derivative having Lipschitz constant $L_p^f$ and $\mathcal{X} = \rset^n$. In this case, according to Example \ref{expl:4}, we can consider the surrogate function: 
	$$g(y;x) = T_p^f(y;x) + \frac{M_p}{(p+1)!} \| y-x\|^{p+1},$$ 
	
	\vspace{-0.2cm}
	
	\noindent 	with $M_p > L_p^f$.  Additionally assume that $x_{k+1}$ is a local minimum of  subproblem 	\eqref{eq:sp}.  Then,  the sequence $(x_{k})_{k\geq0}$ generated by  GHOM satisfies the following convergence rate in terms of first and second-order optimality conditions:
	\vspace*{-0.2cm}
	\begin{align*}
		\min_{i=1:k}  &  \max \!  \left(	\! \left(- \lambda_{\min}^+ (\nabla^2 f(x_{i})) \right)^{\frac{p+1}{p-1}}\!,  	
		p^{\frac{p+1}{p}}  \lambda_{\max}^{\frac{p+1}{p-1}} (B) \!  \left( \! \frac{M_p +L_p^f}{(p-1)!} \! \right)^{\frac{p+1}{p(p-1)}} \!\! \| \nabla f(x_{i})\|_{*}^{\frac{p+1}{p}} \! \right) \\ 
		& \leq  \frac{p(p+1)}{((p-1)!)^{\frac{2}{p-1}}} \cdot  \frac{(M_p+L_p^f)^{\frac{p+1}{p-1}}}{M_p -L_p^f} \cdot (f(x_0) - f^*) \cdot  \frac{\lambda_{\max}^{\frac{p+1}{p-1}} (B)}{k}.
	\end{align*}
\end{theorem}

\begin{proof}
	From the descent property \eqref{eq:descGhom}, we have: $g(x_k; x_k) \geq g(x_{k+1};x_k)$. Further, taking into account the  property $(iii)$ of the  surrogate function, we get:
	\vspace*{-0.2cm}
	\begin{align*} 
		f(x_k)  \geq f(x_k) +\sum_{i=1}^{p} \frac{1}{i !} D^{i} f(x_k)[x_{k+1}-x_k]^{i} + \frac{M_p}{(p+1)!} \|x_{k+1}-x_k\|^{p+1}.  \end{align*}
	
	\vspace{-0.3cm}
	
	\noindent If we define $r_k = \|x_{k+1}-x_k\|$, then we further get: 
	\vspace*{-0.2cm}
	\begin{equation}
		\label{eq:11}
		-\frac{M_p}{(p+1)!} r_{k+1}^{p+1} \geq \sum_{i=1}^{p} \frac{1}{i !} D^{i} f(x_k)[x_{k+1}-x_k]^{i}.
	\end{equation} 
	
	\vspace{-0.2cm}
	
	\noindent In  view of relation \eqref{eq:TayAppBound}, we have:
	\vspace*{-0.2cm}
	\begin{align*}
		& f(x_{k+1})  \leq T_p^f(x_{k+1}; x_k) + \frac{L_p^f}{(p+1)!} r_{k+1}^{p+1} \\
		& \leq f(x_k) +\! \sum_{i=1}^{p} \! \frac{1}{i !} D^{i} f(x_k)[x_{k+1} \!-x_k]^{i} \!+ \frac{L_p^f}{(p+1)!} r_{k+1}^{p+1}  \!\overset{\eqref{eq:11}}{\leq}\! f(x_k) \!- \frac{M_p - L_p^f}{(p+1)!} r_{k+1}^{p+1}. 	
	\end{align*}
	
	\vspace{-0.2cm}
	
	\noindent Hence, we obtain the following descent relation:
	\vspace*{-0.2cm}
	\begin{equation} \label{eq:14}
		f(x_k) - f(x_{k+1}) \geq \frac{M_p -L_p^f}{(p+1)!} r_{k+1}^{p+1}.
	\end{equation}
	Further, from the optimal conditions for $x_{k+1}$ we obtain:
	\begin{equation}\label{eq:9}
		\nabla g(x_{k+1};x_k) = \nabla T_p^f(x_{k+1};x_k) +  \frac{M_p}{p!}\| x_{k+1}-x_k\|^{p -1} B (x_{k+1} -x_k) = 0. 
	\end{equation}
	Using inequality \eqref{eq:TayAppG1} for $f$, we get:
	\vspace*{-0.2cm}
	\begin{equation*}
		\| \nabla f(x_{k+1}) - \nabla T_p^f(x_{k+1}; x_k) \|_{*} \leq \frac{L_p^f}{p!} \| x_{k+1} - x_k\|^p.
	\end{equation*}
	This yields  the following relation:
	\vspace*{-0.2cm}
	\begin{align*}
		\| \nabla f(x_{k+1})\|_{*} &\leq \|\nabla T_p^f(x_{k+1}, x_k) \|_{*} + \frac{L_p^f}{p!} r_{k+1}^p \\
		& \overset{\eqref{eq:9}}{\leq} \| -\frac{M_p}{p!} r_{k+1}^{p-1} B (x_{k+1} - x_k)\|_{*} + \frac{L_p^f}{p!} r_{k+1}^{p}\leq \frac{M_p + L_p^f}{p!} r_{k+1}^p
	\end{align*}	
	or equivalently 
	\vspace*{-0.3cm}
	\begin{equation}\label{eq:10}
		\left(\frac{p!}{M_p + L_p^f}\right)^{\frac{p+1}{p}}	\| \nabla f(x_{k+1})\|_{*}^{\frac{p+1}{p}} \leq r_{k+1}^{p+1}.
	\end{equation}
	Moreover, since  $x_{k+1}$ is a local minimum of $g(\cdot; x_k)$, we have:
	$ \nabla^2 g(x_{k+1}; x_{k}) \succcurlyeq 0.	$
	Computing explicitly the above expression, we obtain:
	\vspace*{-0.2cm}
	\begin{equation} \label{eq:13}
		\nabla^2 T_p^f(x_{k+1}; x_k) + \frac{M_p(p-1)}{p!} r_{k+1}^{p-3} B (x_{k+1} -x_k) (x_{k+1} -x_k)^T B  + \frac{M_p}{p!} r_{k+1}^{p-1} B \succcurlyeq 0.
	\end{equation}
	However, from \eqref{eq:TayAppG2} applied to $f$, we have:
	\vspace*{-0.3cm}
	\begin{equation*}
		\nabla^2 f(x_{k+1}) +\frac{L_p^f}{(p-1)!} r_{k+1}^{p-1} B \succcurlyeq \nabla^2 T_p^f(x_{k+1};x_k).
	\end{equation*}
	Returning with the above relation in \eqref{eq:13} we get:
	\vspace*{-0.2cm}
	\begin{equation*}
		\nabla^2 f(x_{k+1}) +\frac{M_p + pL_p^f}{p!} r_{k+1}^{p-1} B + \frac{M_p(p-1)}{p!} r_{k+1}^{p-3} B (x_{k+1} -x_k) (x_{k+1} -x_k)^T B \succcurlyeq 0.
	\end{equation*}
	Rearranging the terms, we obtain:
	\vspace*{-0.2cm}
	\begin{align*}
		- \nabla^2 f(x_{k+1}) & \preccurlyeq \frac{M_p + pL_p^f}{p!} r_{k+1}^{p-1} B + \frac{M_p(p-1)}{p!} r_{k+1}^{p-3}B (x_{k+1} -x_k) (x_{k+1} -x_k)^T B \\
		&\preccurlyeq  \frac{M_p + pL_p^f}{p!} r_{k+1}^{p-1} B + \frac{M_p(p-1)}{p!} r_{k+1}^{p-3} r_{k+1}^{2} B \preccurlyeq \frac{M_p+L_p^f}{(p-1)!} r_{k+1}^{p-1} B.
	\end{align*}
	Tacking the maximum eigenvalue, we obtain:
	\vspace*{-0.2cm}
	\begin{align*}
		\!	\lambda_{\max} (\!-\!\nabla^2 \!f(x_{k+\!1})\!) \!\leq\!\! \frac{M_p\!+\!L_p^f}{(p\!-\!1)!} r_{k+1}^{p-1} \lambda_{\max} (B) \, \text{or}  	-\!\!\lambda_{\min} (\nabla^2 \!f(x_{k+\!1})\!) \!\!\leq\!\! \frac{M_p\!+\!L_p^f}{(p\!-\!1)!} r_{k+1}^{p-1} \lambda_{\max} (B).
	\end{align*}

	\vspace{-0.3cm}
	
\noindent 	Finally, using the notation $a^+ = \max(a, 0)$ and the fact that if $a \leq b$ for some $b \geq 0$, then also $a^+ \leq b$,  the above inequality yields:
	
	\vspace{-0.6cm}
	
	\begin{equation}\label{eq:15}
		\left(\frac{(p-1)!}{(M_p+L_p^f)  \lambda_{\max} (B)}\right)^{\frac{p+1}{p-1}}  (- \lambda_{\min}^+ (\nabla^2 f(x_{k+1})))^{\frac{p+1}{p-1}} \leq  r_{k+1}^{p+1}.
	\end{equation}

	\vspace{-0.2cm}
	
	\noindent	By combining \eqref{eq:10} and \eqref{eq:15} we get the following compact form:
	
	\vspace{-0.6cm}

	\begin{align*}
		& \zeta_{k+1}  := \max  \Bigg(	\left(\frac{(p-1)!}{(M_p+L_p^f)  \lambda_{\max} (B)}\right)^{\frac{p+1}{p-1}} (- \lambda_{\min}^+ (\nabla^2 f(x_{k+1})))^{\frac{p+1}{p-1}},\\ 
		& \qquad \qquad \qquad  \left(\frac{p!}{M_p + L_p^f}\right)^{\frac{p+1}{p}}	\| \nabla f(x_{k+1})\|_{*}^{\frac{p+1}{p}} \Bigg)  \leq r_{k+1}^{p+1}.
	\end{align*} 

	\vspace*{-0.2cm}
	
	\noindent	Telescoping  \eqref{eq:14}, we get:
		\vspace*{-0.2cm}
	$$ f(x_0) - f^* \geq \frac{M_p -L_p^f}{(p+1)!} \sum_{i=0}^{k-1}r_{i+1}^{p+1} \geq \frac{M_p -L_p^f}{(p+1)!} \sum_{i=0}^{k-1} \zeta_{i+1} \geq \frac{k(M_p -L_p^f)}{(p+1)!} \min_{i=1:k} \zeta_{i}. $$
	
		\vspace{-0.2cm}
	
	\noindent Rearranging the terms we get the statement of the theorem.
\end{proof}

\noindent Note that   \cite{BriGar:17, CarGou:11, cartis:2020} obtains similar convergence results for their Taylor-based algorithsm as in Theorem \ref{th:soc}. However,   $x_{k+1}$ in these papers  must  satisfy some second-order optimality conditions. E.g., in   \cite{CarGou:11} the authors require: $   \max \left (0, \!-\lambda_{\min}(\nabla^2 g(x_{k+1}; x_k)\right) \!\leq\! \theta \| x_{k+1} - x_k \|^{p-1} $ for some $\theta \!>0$ and $\| \nabla g(x_{k+1}; x_k)\| \! \leq\! \theta \| x_{k+1} - x_k\|^p $.

\section{Conclusions}
This paper has derived a unified  analyzis of  the convergence behaviour  of  higher-order majorization-minimization algorithms for minimizing convex/nonconvex functions that admit a surrogate model such that the corresponding error function has a $p \geq 1$ higher-order Lipschitz continuous derivative.  Under these settings we have derived global  convergence results for our algorithm in terms of function values or  first-order optimality conditions. Faster rates of convergence were established  under uniform convexity or KL property of the objective function.  One can also notice that although the conditions for the surrogate from Assumption \ref{def:sur} were sufficienttly powerful to derive convergence results for general (non)convex problems,  for deriving  convergence rates in  some first-order optimality criterion an additional condition \eqref{eq:unifh_domf} on the surrogate was required in the nonconex setting. Finally, for unconstrained  nonconvex problems we have derived convergence rates in terms   of  first- and second-order optimality conditions.   Hence, our algorithmic framework led to  an elegant  and novel convergence analysis, yielding new results or covering results  that are  otherwise scattered across a dozen  works.

\section*{Acknowledgments}
	\noindent The research leading to these results has received funding from the NO Grants  RO-NO-2019-0184, under project ELO-Hyp, nr. 24/2020; UEFISCDI PN-III-P4-PCE-2021-0720, under project L2O-MOC, nr. 70/2022.  The authors would like to thank Yurii Nesterov for in-sightful discussions. 

\bibliographystyle{siamplain}
\bibliography{references}

\end{document}